\documentclass[a4paper,10pt]{article}

\usepackage{amsmath}
\usepackage{amscd}
\usepackage{amsfonts}
\usepackage{stmaryrd}
\usepackage{mathrsfs}
\usepackage{amsthm}
\usepackage{dsfont}
\usepackage{latexsym, amssymb,amsthm,amsmath,lscape,epsfig,setspace,tikz,slashed}
\usepackage[retainorgcmds]{IEEEtrantools}
\usetikzlibrary{decorations.markings}
 \usetikzlibrary{arrows}
\usepackage[pdftex]{hyperref}
\usepackage[margin=10pt,font=small,labelfont=bf]{caption}
\usepackage{a4wide}
\usepackage[all,cmtip]{xy}
\usepackage{multicol}

\newcommand{\mfg}{\mathfrak{g}}

\newcommand{\R}{\mathbb{R}}

\newcommand{\J}{\mathbb{J}} 

\newcommand{\Id}{\operatorname{Id}}

\newcommand{\T}{\mathbb{T}}

\newcommand{\C}{\mathbb{C}}

\newcommand{\mcL}{\mathcal{L}}
\newcommand{\mcH}{\mathcal{H}}

\newcommand{\mcO}{\mathcal{O}}

\newcommand{\mcG}{\mathcal{G}}
\newcommand{\lb}{\llbracket}
\newcommand{\rb}{\rrbracket}

\newcommand{\del}{\partial}
\newcommand{\delbar}{\bar{\partial}}

\newcommand{\tL}{\widetilde{L}}

\newcommand{\ve}{\varepsilon}

\newcommand{\mcC}{\mathcal{C}}
\newcommand{\mcD}{\mathcal{D}}
\newcommand{\floor}[1]{\lfloor #1 \rfloor}
\newcommand{\anchor}{\pi_T}

\newcommand{\Cour}[1]{\llbracket #1\rrbracket} 

\newcommand{\II}{\mathbb{I}}
\newcommand{\JJ}{\mathbb{J}}

\newcommand{\dsum}{\oplus}
\newcommand{\iso}{\cong}
\renewcommand{\^}{\wedge}
\newcommand{\rank}{\textrm{rank}}
\renewcommand{\epsilon}{\varepsilon}
\newcommand{\Lie}{\mathcal{L}}

\newcommand{\mf}[1]{ \mathfrak{#1}}

\numberwithin{equation}{section}

\theoremstyle{plain}
\newtheorem{thm}{Theorem}[section]
\newtheorem{thm*}{Theorem}
\newtheorem{prop}[thm]{Proposition}

\theoremstyle{definition}
\newtheorem{lem}[thm]{Lemma}
\newtheorem{defn}[thm]{Definition}
\newtheorem{defn/thm}[thm]{Definition/Theorem}
\newtheorem{ex}[thm]{Example}
\newtheorem{cor}[thm]{Corollary}

\theoremstyle{remark}
\newtheorem{rem}[thm]{Remark}

\begin{document}

\title{A neighbourhood theorem for submanifolds in generalized complex geometry}

\author{M.A.Bailey \thanks{{\tt michael.bailey@uwaterloo.ca}},  University of Waterloo \\ \ \\
G.R.Cavalcanti \thanks{{\tt G.R.Cavalcanti@uu.nl. Supported by the VIDI grant 639.032.221 from NWO, the Netherlands Organisation for Scientific Research.}} , Utrecht University\\ \ \\
J.L.van der Leer Dur\'an \thanks{{\tt joeyvdld@math.toronto.edu. Supported by the NSERC Discovery Grant 355576.}}, University of Toronto
}
\date{\vspace{-5ex}}

\maketitle

\abstract \noindent  
We study neighbourhoods of submanifolds in generalized complex geometry. Our first main result provides sufficient criteria for such a submanifold to admit a neighbourhood on which the generalized complex structure is $B$-field equivalent to a holomorphic Poisson structure. This is intimately tied with our second main result, which is a rigidity theorem for generalized complex deformations of holomorphic Poisson structures. Specifically, on a compact manifold with boundary we provide explicit conditions under which any generalized complex perturbation of a holomorphic Poisson structure is $B$-field equivalent to another holomorphic Poisson structure. The proofs of these results require two analytical tools: Hodge decompositions on almost complex manifolds with boundary, and the Nash-Moser algorithm. As a concrete application of these results, we show that on a four-dimensional generalized complex submanifold which is generically symplectic, a neighbourhood of the entire complex locus is $B$-field equivalent to a holomorphic Poisson structure. Furthermore, we use the neighbourhood theorem to develop the theory of blowing down submanifolds in generalized complex geometry.   
 

\vskip12pt

\tableofcontents


\section{Introduction}
	A fundamental question about any geometric structure is ``What does it look like?". This can be asked at different levels, such as in a neighbourhood of a point or in a neighbourhood of a special submanifold. In a neighbourhood of a point, an answer to this question  amounts to finding all local invariants of the structure.
	For example, complex and symplectic structures have no local invariants and any point has a neighbourhood equivalent to a fixed standard model.
	In contrast, Riemannian metrics and Poisson structures do not admit such local models and further hypothesis are needed to produce a meaningful statement, see, for example, \cite{MR0257932,MR0355902} for the Riemannian case and  \cite{MR794374,MR2776372} for the Poisson case.
	
	Similar results describing the neighbourhood of a special submanifold are harder to come by as typically there is more local data to influence the behaviour of the structure.
	In symplectic geometry, Weinstein's Lagrangian and the Symplectic Neighbourhood Theorems \cite{MR0286137} provide successful examples where one can fully describe the geometric structure in a neighbourhood of a special submanifold.
	In contrast, in complex geometry similar results can only be proved under more restrictive assumptions  \cite{MR0137127,MR0206980}.

	We are interested in the local structure of generalized complex structures. For these structures, the state-of-the-art on local form theorems are those for neighbourhoods of branes and for the type change locus on stable generalized complex manifolds \cite{MR3805052},  for Poisson transversals \cite{MR3894047,2016arXiv160505386B}  and Bailey's theorem for neighbourhoods of points \cite{MR3128977}.
	The results in \cite{MR3805052,MR3894047,2016arXiv160505386B} rely on symplectic tools and are similar to Weinstein's Lagrangian and Symplectic Neighbourhood Theorems.
	The normal form theorem for points proved in \cite{MR3128977} is complex in nature and states that
%
%
generalized complex structures are locally equivalent to the product of symplectic and holomorphic Poisson structures.

	Our objective here is to extend the results from \cite{MR3128977} and provide a local form theorem for the neighbourhood of special types of submanifolds. Since we are aiming for a result that is complex in nature, the first desired property of these special submanifolds is that the generalized complex structure in a neighbourhood of the submanifold can be compared with a reference complex structure.  {\it Abelian Poisson Branes} (see Definition \ref{Poisson brane}) are a natural class of submanifolds for which such comparison can be made and, similar to Bailey's result, we show that there is an equivalence between smooth and holomorphic objects. A precise statement is given in  Theorem \ref{16:02:38}, but the message is that  if a neighbourhood $U$ of an Abelian Poisson brane $Y$ is sufficiently convex and the Dolbeault cohomology $H^{0,2}(U)$ vanishes, then the generalized complex structure  is $B$-field equivalent to a holomorphic Poisson structure on a neighbourhood of $Y$. 

This result will follow from an openness result for holomorphic Poisson structures. Namely, in Theorem \ref{14:11:52} we prove that if $(M,I)$ is a holomorphic Poisson manifold  whose boundary is sufficiently convex  and $H^{0,2}(M,I)=0$, then any nearby generalized complex structure is also holomorphic Poisson. This result is new even for compact manifolds without boundary (in which case the convexity hypothesis is empty) and implies, for example, that generalized complex structures near complex structures are of holomorphic Poisson type.
	
	The main tools to prove these results are analytical, namely, we use a Nash--Moser type of argument as developed in \cite{MR656198} (which relies on some version of elliptic regularity/Hodge theory), and Hodge theory for complex manifolds with boundary developed recently by van der Leer Duran \cite{s00208-021-02293-5} (which relies on the geometric condition of  having sufficiently convex boundary).
		
	Applications of this theorem arise readily in four real dimensions, where the hypotheses of convexity and vanishing of cohomology are easier to be fulfilled. The first application concerns the behaviour of the complex locus. We prove:
\begin{thm*}\label{theo:intro the 1}
	Suppose $Y$ is the compact complex locus inside a generalized complex 4-manifold $(M,\JJ)$. Then $\JJ$ is equivalent to a holomorphic Poisson structure in a neighbourhood of $Y$ in $M$.
\end{thm*}

The second application is to establish a converse to the blow-up procedure introduced in \cite{MR3894047}. We study the question of when a generalized complex manifold is the (canonical) blow-up of another generalized complex manifold. The full answer is given in Theorem \ref{theo:blow-down}, and, in four real dimensions, the result becomes:
	\begin{thm*}\label{theo:intro the 1}
	Let $Y$ be a real two-dimensional surface in the complex locus of a four-dimensional generalized complex manifold $\widetilde M$. If $Y$ is diffeomorphic to $\C\mathbb{P}^1$ and has self-intersection $-1$, then $Y$ can be blown down to a point, that is, $M = \widetilde M /Y$ admits a generalized complex structure for which the quotient map $\widetilde M \to M$ is generalized holomorphic. 
	\end{thm*}

While we focused our attention on the generalized complex implications of our main theorems, we expect that their usefulness will go beyond their immediate area. For example, since the introduction of generalized complex structures two things became apparent:
\begin{enumerate}
\item these structures provide a good mathematical framework to describe type II string theory,
\item physicists and mathematicians do not have adequate tools to deal with the type change locus of a generalized complex manifold, which is precisely where the most interesting behaviour happens.
\end{enumerate}
The theory developed here provides the missing tools and they are as nice as we could hope for: the type change locus is described by holomorphic data.

This paper is organized as follows. In Section \ref{GCG} we review the basics of generalized complex geometry, starting with the geometry of $TM \oplus T^*M$ in \ref{sec:TTM}, following with generalized complex structures in Section \ref{sec:gcs} and finally introducing Abelian Poisson branes in Section \ref{sec:ABP}. In Section \ref{14:20:32} we introduce the framework for deformations of generalized complex structures, focusing on structures of holomorphic Poisson type and in Section \ref{sec:17Jan:0027} we prove our first result on stability of holomorphic Poisson structures on compact manifolds without boundary (Theorem \ref{15:13:30}). We continue in Section \ref{sec:hodge} with Hodge theory, the tool needed to state and prove the same result on manifolds with boundary. In Section \ref{sec:main results} we can finally state our main theorems. We follow these statements with applications in Section \ref{Applications} and finish the paper with the proofs of our main results, in the last section of the paper.

\section{Generalized complex geometry}\label{GCG}

Our main objective in this paper is to prove a neighbouhood theorem for generalized complex submanifolds of complex type. In this introductory section we briefly review the basic notions from generalized complex geometry relevant for this work and introduce the submanifolds on which we will focus our attention. For more details on the topic covered and proofs of the statements we refer to \cite{MR2811595}. 

\subsection{The double tangent bundle}\label{sec:TTM}
Given an $m$-dimensional manifold, $M$, equipped with a closed three-form, $H \in \Omega^3(M)$, one can form the \textsl{double tangent bundle} $\T M:=TM\oplus T^\ast M$. Elements of $\T M$ are denoted by $X+\xi, Y+\eta, \ldots$, where $X,Y\in TM$ and $\xi,\eta\in T^\ast M$, or simply by $u,v,\ldots,$ if the distinction between vectors and forms is not necessary. There are three relevant structures on $\T M$. The first is the \textsl{natural pairing} given by the pointwise evaluation of forms on vectors:
\begin{align*}
\langle X+\xi,Y+\eta	\rangle:=\frac{1}{2}\big(\xi(Y)+\eta(X)\big).
\end{align*}
The natural pairing is non-degenerate of signature $(m,m)$, and both $TM$ and $T^\ast M$ are isotropic, that is, the natural paring vanishes when restricted to these subbundles.

The second is the natural projection map $\anchor\colon\T M\rightarrow TM$, also called the \textsl{anchor}.

The third is a bracket on the space of sections, the \textsl{Courant bracket} \cite{MR998124}, which is a natural lift via the anchor of the Lie bracket of vector fields to sections of $\T M$:
\begin{align*}
\lb X+\xi,Y+\eta	\rb:=[X,Y]+\mcL_X\eta-\iota_Yd\xi-\iota_Y\iota_XH.
\end{align*}
Sometimes we write $\lb \cdot , \cdot \rb_H$ to emphasize which three-form is being used.  The following lemma lists the main properties of $\T M$ with these structures. 
\begin{lem}\label{19:16:48} For $u,v,w\in C^\infty(\T M)$ and $f\in C^\infty(M)$ we have
\begin{itemize}
\item[i)] $\lb u, \lb v,w \rb \rb=\lb \lb u, v\rb ,w  \rb+\lb v, \lb u,w \rb \rb$,
\item[ii)] $\anchor (\lb u,v \rb)=[ \anchor(u),\anchor(v)]$,					
\item[iii)] $\lb u,fv\rb=f\lb u,v\rb+(\anchor(u)\cdot f)v $, 
\item[iv)] $\lb u,u \rb=d\langle u,u \rangle$,  
\item[v)] $\langle \lb u,v \rb,w \rangle +\langle v, \lb u,w \rb \rangle=\anchor(u)\cdot \langle v,w	\rangle$.
\end{itemize}
\end{lem}
\begin{defn}
A vector bundle $E$ over $M$ equipped with a non-degenerate pairing $\langle \cdot, \cdot \rangle$, an anchor $\anchor\colon E\rightarrow TM$ and a bracket $\lb\cdot, \cdot\rb$ satisfying the above axioms is called a \textsl{Courant algebroid}.
\end{defn}

Among all Courant algebroids, $\T M$ has the special property that the anchor map $\T M \stackrel{\anchor}{\to} TM$ and its adjoint $T^*M  \stackrel{\anchor^*}{\to} \T M$ together give rise to a short exact sequence:
$$
0 \to T^*M \stackrel{\anchor^*}{\to}  \T M \stackrel{\anchor}{\to} TM \to 0.
$$

\begin{defn}
An {\it exact Courant algebroid}   is  a Courant algebroid $E\to M$ for which
\begin{equation}\label{eq:exact courant}
0 \to T^*M \stackrel{\anchor^*}{\to}  E\stackrel{\anchor}{\to} TM \to 0
\end{equation}
is an exact sequence.
\end{defn}

Given an exact Courant algebroid $E \to M$, a choice of isotropic splitting $s\colon TM \to E$ gives rise to an isomorphism of Courant algebroids between $E$ and $\T M$. From this point of view, the immersion $T^* M \to \T M$ is natural, but the immersion $TM \to \T M$ is not. Also the three-form $H$ is related to the lack of integrability of the splitting chosen.  While we will continue to work with $\T M$, the framework of exact Courant algebroids clarifies some of the aspects of the theory we describe below.

An important difference between $\T M$ and $TM$ is that these spaces have different symmetries. For example, while any diffeomorphism $\varphi\colon M\to M$ preserves the Lie bracket of vector fields, i.e.,
\[
\varphi_*[X,Y] = [\varphi_*X,\varphi_*Y],
\]
the same diffeomorphism will only relate the brackets $\lb \cdot,\cdot\rb_H$ and $\lb \cdot,\cdot\rb_{\varphi_*H}$:
\[
\varphi_*(\lb u,v \rb_H)= \lb\varphi_*u,\varphi_*v\rb_{\varphi_*H},
\]
where the action of $\varphi$ on $\T M$ is given by\footnote{The splitting in this matrix refers to the splitting $\T M=TM\oplus T^\ast M$.} 
\begin{align*}
\varphi_\ast:=\begin{pmatrix} \varphi_\ast & 0\\ 0 & (\varphi^{-1})^\ast \end{pmatrix}\colon \T M\rightarrow \T M,
\end{align*}
On the other hand, $\T M$ has further symmetries not present on $TM$ which are given by the action of two-forms. Indeed, a two-form $B\in\Omega^2(M)$ acts on $\T M$ via:
\begin{align*}
e^B (X+\xi):=X+\xi+\iota_XB.
\end{align*}
From the point of view of exact Courant algebroids, the action of two-forms corresponds to different choices of splittings of \eqref{eq:exact courant}. As for diffeomorphisms, the action of two-forms preserves the natural pairing but only relates different Courant brackets:
\[
e^B (\lb u,v \rb_H)= \lb e^B u,e^B v\rb_{H + dB},
\]

Even though separately the actions of diffeomorphisms and two-forms do not necessarily preserve the Courant bracket, the combination $\varphi_\ast\circ e^{B}$ does, provided that $\varphi^\ast H=H+dB$. This expresses the symmetries of $\T M$ as an extension of the diffeomorphisms preserving the cohomology class $[H]$ by closed two-forms:
$$0\to \Omega^2_{cl}(M) \to \mathrm{Sym}(\T M) \to \mathrm{Diff}_{[H]}(M) \to 0.$$
In particular closed two-forms always provide symmetries of $\T M$ and the action of a closed 2-form is called a \emph{$B$-field transform}.

Just as one can exponentiate vector fields to produce a one parameter family of diffeomorphisms, one can exponentiate sections of $\T M$ to produce symmetries of $\T M$. The defining equation for the \textsl{flow} of a section $u \in C^\infty(\T M)$ is 
\begin{align*}
 \frac{d}{dt} F_{tu}(v)=-\lb u ,F_{tu}(v) \rb.
\end{align*}
We can solve this ordinary differential equation explicitly in terms of the flow of vector fields. Indeed, given $u = X+\xi\in C^\infty(\T M)$ the flow of $u$ is the one-parameter family of symmetries of $\T M$, given by 
\begin{align}
F_{tu}:=
{\varphi_t}_\ast  e^{B_t}:\T M\rightarrow \T M,
\label{10:12:54}
\end{align}
where $\varphi_t$ is the flow of $X$ and $B_t:=\int_0^t\varphi_s^\ast(d\xi+\iota_{X}H)ds$.

\subsection{Generalized complex structures}\label{sec:gcs}

As before, let $(M,H)$ be a smooth $m$-manifold equipped with a closed three-form.  
\begin{defn}\label{13:35:50}
A \textsl{generalized complex structure} on $(M,H)$ is a complex structure $\J$ on $\T M$ which is orthogonal with respect to the natural pairing and whose $(+i)$-eigenbundle, $L\subset\T M_\C$, is involutive, i.e.\ $\lb C^\infty(L),C^\infty(L)\rb \subset C^\infty(L)$. 
\end{defn}

 An orthogonal complex structure $\J$ on $\T M$  is an \textsl{almost generalized complex structure} and involutivity of $L$ is refered as the \textsl{integrability} condition. 

 Orthogonality of $\J$ is equivalent to isotropy of $L$, which is therefore a Lagrangian subbundle of $\T M$ since  the decomposition $\T M_\C=L\oplus \overline{L}$ forces $L$  to be of maximal dimension. There is, therefore, a one-to-one correspondence between almost generalized complex structures and Lagrangian subbundles, $L \subset \T M$, that satisfy the \textsl{non-degeneracy condition} $L\cap \overline{L}=0$.  Involutive Lagrangian subbundles of $\T M_\C$ are also called \textsl{(complex) Dirac structures}, hence there is a one-to-one correspondence between generalized complex structures and nondegenerate complex Dirac structures. 
 
One of the most basic (pointwise) invariants of an almost generalized complex structure, $\J$, is its type:
\[
\mathrm{type}(\J)(p) := \dim_\C(T^*_pM \cap \J T^*_pM), \qquad p \in M.
\]
Since the type is the dimension of the intersection of two subbundles of $\T M$, it is a upper semicontinous function.

An alternative description of the type of $\J$ is obtained by considering the composition
$$T^*M \stackrel{\anchor^*} \to \T M \stackrel{\J}{\to} \T M \stackrel{\anchor}{\to} T M.$$
Orthogonality of $\J$ implies that $\anchor \circ \J \circ \anchor^*\colon  T^*M \to TM$ is induced by a bivector $\pi_\J$ and the type of $\J$ is (half of) the corank of $\pi$. If $\J$ is integrable, $\pi_\J$ is Poisson \cite{crainic2011}.

The next few examples indicate that the type corresponds to the number of complex directions $\J$ has at any given point. 

\begin{ex} 
\label{11:37:30a}
Let $I$ be an almost complex structure on $M^{2n}$. Then 
\begin{align}\label{14:18:03}
\J_I:=
\begin{pmatrix}
-I & 0 \\
0 & I^\ast
\end{pmatrix}
\end{align}
defines an almost generalized complex structure whose type is $n$ everywhere. The corresponding Dirac structure is given by $L_I=T^{0,1}M\oplus T^{\ast 1,0}M$. 
Integrability of $\J_I$ is equivalent to integrability of $I$ and that $H$ be of type $(2,1)+(1,2)$ with respect to $I$.

Conversely, if a generalized complex structure, $\J$, has type $n$ everywhere, then $\J: T^*M \to T^*M$ and hence defines an almost complex structure on $M$. Further, one can show that there is a two-form $B$ that transforms $\J$ into a complex structure, that is $\J = e^B_\ast \J_I e^{-B}_\ast$ for some $B \in \Omega^2(M)$. Therefore complex structures provide the prototype for structures of type $n$.
\end{ex}

\begin{ex}\label{11:37:30b}
Let $\omega \in \Omega^2(M)$ be a nondegenerate two-form, which therefore provides an isomorphism $\omega: TM \to T^*M$. Then
\begin{align}
\J_\omega:=
\begin{pmatrix}
0 & -\omega^{-1} \\
\omega & 0
\end{pmatrix}
\label{13:31}
\end{align}
is an almost generalized complex structure and since $\J:T^*M \to TM$, the structure has type zero everywhere. The associated Dirac structure is given by $L_\omega=e^{-i\omega}_\ast (TM_\C)=\{X-i\omega(X)| \ X\in TM_\C \}$. 
 The structure $\J_\omega$ is integrable if and only if $\omega$ is closed and $H=0$.
 
Conversely, if $\J$ has type zero everywhere, then both $TM$ and $\J T^*M $ are Lagrangian spaces complimentary to $T^*M$ hence there is a 2-form $B$ for which $\J T^*M= e^B TM$. Using this form to transform $\T M$, we put $\J$ in the block form above, that is $\J = e^B \J_\omega e^{-B}$. If $\J$ is integrable we also have that $dB = H$ and $d\omega =0$.
\end{ex}

\begin{ex}\label{11:37:30c}
Let $(I,\sigma)$ be a holomorphic Poisson structure on $M$ which we decompose into real and imaginary parts as $\sigma:=-\frac{1}{4}\left(IP + iP\right)$, that is, $P = -4\operatorname{Im}(\sigma)$. If $H=0$ then 
\begin{align}
\J_{(I,\sigma)}:=
\begin{pmatrix}
-I & P \\
0 & I^\ast
\end{pmatrix}
\label{13:30}
\end{align}
defines a generalized complex structure. The corresponding Dirac structure is given by 
\begin{align}\label{3:43:58}
L_{(I,\sigma)}:=T^{0,1}M\oplus e^\sigma(T^{\ast 1,0}M)=\{X+\sigma(\xi)+\xi| \ X\in T^{0,1}M, \ \xi\in T^{\ast 1,0}M\}.
\end{align} 
The type coincides with $\frac{1}{2}\text{corank}(P)$.

The converse to this result also holds. Namely, if the component of $\J$ that maps $TM$ to $T^*M$ vanishes, then $\J$ has the form \eqref{13:30} which allows us to define a complex structure and read off the imaginary part of the holomorphic Poisson bivector which in together fully determine the holomorphic Poisson bivector.  Stated in a splitting-invariant way, given a generalized complex structure $\J$, an involutive, isotropic splitting $s:TM \to \T M$ for the anchor $\pi:\T M\rightarrow TM$ whose image is invariant by $\J$ induces a natural holomorphic Poisson structure on $M$. Indeed, in this situation there is a unique $B$-field transform for which $e^B s(TM) =  TM$ and, by invariance of $TM$ under $e^{B}\J e^{-B}$ has the form \eqref{13:30}.
\end{ex}

The situation described in Example \ref{11:37:30c} occurs frequently enough to deserve a name.

\begin{defn}
A \emph{holomorphic gauge} for a generalized complex structure $\JJ$ is a $B$-field $B$ such that
\begin{align}\label{HP1}
e^{B} \JJ e^{-B} =
\begin{pmatrix}
-I & P \\
0 & I^*
\end{pmatrix}
\end{align}
for some $I$ and $P$. In this case, $I$ is a complex structure and $\sigma := -\frac{1}{4}\left(IP + iP\right)$ is a holomorphic Poisson structure.
\end{defn}


Examples \ref{11:37:30a} to \ref{11:37:30c} tie up generalized complex structures with the more common complex and symplectic structures. A valid question is whether there are examples which are not modeled on complex nor symplectic objects. The answer to the question is interesting as there is a marked difference between local and global behaviour. As we mentioned before, locally these examples extinguish all possibilities:

\begin{thm}[Bailey \cite{MR3128977}]\label{theo:Bailey}
A point  of type $k$ in a generalized complex manifold has a neighbourhood equivalent, via the action of diffeomorphisms and two-forms, to a neighbourhood of $0$ in $\R^{2(n-k)}\times \C^k$, where $\R^{2(n-k)}$ is endowed with the standard symplectic structure and $\C^k$ with a holomorphic Poisson structure which vanishes at $0$.
\end{thm}

That is, locally a generalized complex manifold is just a product of the structures introduced in Examples \ref{11:37:30a} to \ref{11:37:30c}. However, it is not always possible to patch these local forms in a compatible way and there are many generalized complex manifolds that do not admit complex or symplectic structures nor are products of these. Concrete examples are given by the connected sums $m\C\mathbb{P}^2\# n\overline{\C\mathbb{P}}^2$: they admit complex or symplectic structures if and only if $m=1$ while  they carry a generalized complex structure precisely when $m$ is odd, which is precisely the condition for admitting an almost complex structure \cite{MR2574746}.

\subsection{Generalized Poisson submanifolds and branes}\label{sec:ABP}

Having introduced generalized complex structures last section, our next task is to introduce the different notions of submanifolds and eventually specialize to the type of submanifolds for which our neighbourhood theorem applies, namely, those whose local behaviour contains enough holomorphic data.

The first notion of submanifold arises naturally in the context of pullbacks of Dirac structures as introduced by Courant \cite{MR998124}. Given a Dirac structure $D$ on a manifold $M$ and a submanifold $i:Y \to M$, for every point $p\in Y$ we define a subspace of $\T_p Y$ by
\begin{equation}\label{eq:pullback}
i^* D_p = \{X + i^*\xi \in \T_p Y| i_*X + \xi \in D_p \}.
\end{equation}
The space $i^*D_p \subset \T_p Y$ is automatically Lagrangian, but the collection of spaces $D_p$ may not vary smoothly from point to point. If it does, involutivity of $D$ with respect to $\lb\cdot,\cdot\rb_H$ implies the involutivity of $i^*D$  with respect to $\lb\cdot,\cdot\rb_{i^*H}$, making it into a Dirac structure on $Y$. One condition that guarantees that $i^*D$ varies smoothly is that $D|_{i(Y)} \cap N^*Y$ has constant rank, where $N^*Y$ is the conormal bundle of $Y$.

\begin{defn}
Given a Dirac structure $D$ on $M$ and a submanifold $i:Y \to M$, the {\it pullback of $D$ to $Y$} is the structure $i^* D$ which is well defined as long as $i^*D$ is a smooth subbundle of $\T Y$. 
\end{defn}

We can rephrase this definition in a splitting independent way. Given a submanifold $i:Y \to M$, we can form on $i^*\T M$ the subbundle $N^*Y^\perp$,  the orthogonal complement of the conormal bundle of $N^*Y \subset i^*\T M$. Then the Courant bracket and the  pairing on $\T M$ induce naturally a Courant bracket and a pairing on $N^* Y^\perp/N^*Y \cong \T Y$. In this description, the pullback Dirac structure is given by
\begin{equation}\label{eq:reduced Dirac}
i^* D  = \frac{D \cap N^* Y^\perp + N^*Y }{N^*Y} \subset \frac{N^* Y^\perp}{N^*Y}.
\end{equation}

Back to the generalized complex world, given a generalized complex manifold $(M,H,\J)$ and a submanifold $i:Y \to M$, under the smoothness condition above, we can pull back the associated Dirac structure, $L$, to $Y$ to obtain a Dirac structure which may or may not be degenerate. If $i^*L$ is nondegenerate, it defines a generalized complex structure on $Y$.

\begin{defn}\label{14:46:48}
A \textsl{generalized complex submanifold} of $(M,H,\J)$ is a submanifold $i:Y \hookrightarrow M$ such that the Dirac pullback $i^*L$ exists and defines a generalized complex structure on $(Y,i^\ast H)$.  
\end{defn}

In the complex and symplectic contexts this notion of submanifold agrees with the notion of complex and symplectic submanifolds, respectively. We now want to specialize to a subclass of generalized complex submanifolds whose behaviour resembles that of complex submanifolds in complex geometry. To this end we will impose three requirements on our generalized complex submanifolds. We first give a short overview with intuitive explanations.

\begin{enumerate}
\item The first requirement is that $\J N^\ast Y=N^\ast Y$, which intuitively means that $\J$ is complex in directions normal to $Y$. Consequently, $Y$ itself inherits a generalized complex structure $\J_Y$ via the isomorphism $\T Y\cong (N^\ast Y)^\perp /N^\ast Y$, a quotient of two $\J$-invariant subspaces.
\item The second requirement is that the induced generalized complex structure $\J_Y$ obtained from step 1.\ is $B$-field equivalent to a holomorphic Poisson structure on $Y$. 
\item The first two requirements imply that $N^{\ast 1,0}Y$ is a complex vector bundle over a complex manifold. The third and final requirement is that $N^{\ast 1,0}Y$ is a holomorphic vector bundle.  
\end{enumerate}

Together, these three conditions intuitively amount to $\J$ being holomorphic in directions tangent and normal to $Y$. Concretely, they imply that there is a natural complex structure on a tubular neighborhood of $Y$ that we can compare against $\J$ itself, and this will play a major role in the rest of the paper. 

To make the above three requirements precise we need some terminology. Below, definitions \ref{GPS}, \ref{GCB} and \ref{APB} correspond to the above requirements 1.,2. and 3., respectively, and the corresponding class of submanifolds will be called Abelian Poisson branes. 

\begin{defn} \label{GPS}
A \textsl{generalized Poisson submanifold} of $(M,H,\J)$ is a submanifold $Y\subset M$ with the property that $\J N^\ast Y=N^\ast Y$.
\end{defn}

Generalized Poisson submanifolds are generalized complex submanifolds in the sense of Definition \ref{14:46:48} and they are Poisson submanifolds for $\pi_\J$. The induced generalized complex structure $\J_Y$ on $Y$ can be understood directly via the isomorphism $\T Y\cong (N^\ast Y)^\perp/N^\ast Y$, which is a quotient of two $\J$-invariant subspaces of $\T M$.

\begin{lem}\label{lem:lie algebra bundle}
Let $i:Y \to (M,H,\J)$ be a generalized Poisson submanifold and let $N^{*1,0}Y$ be the $+i$-eigenspace of $\J$ on $N^*_\C Y$. Then for every $y \in Y$ the space $N^{*1,0}_yY$ inherits a Lie algebra structure whose bracket is given by
$$[\alpha,\beta] = \lb\tilde\alpha,\tilde\beta\rb|_y,$$
where $\tilde\alpha,\tilde\beta \in C^\infty(L)$ are smooth extensions of $\alpha,\beta \in N_y^{*1,0}Y$. 
\end{lem}
\begin{proof}
We only need to check that this bracket is well defined since the Lie bracket properties follow from the corresponding properties of the Courant bracket.

The first step is to check that the bracket does not depend on the choice of extensions $\tilde\alpha$ and $\tilde \beta$. We check that for $\tilde\beta$. A different extension of $\beta$ differs from $\tilde\beta$ by a section  $\tilde\beta'\in C^\infty(L)$ which vanishes at $y$ and we need to show that for such sections $\Cour{\tilde\alpha,\tilde\beta'} =0$  at $y$. Using bilinearity of the Courant bracket we may assume  that $\tilde\beta' = f \gamma$, where $f$ is a function that vanishes at $y$ and $\gamma \in C^\infty(L)$.  For such a section we have
$$\Cour{\tilde\alpha,f \gamma} = f\Cour{\tilde\alpha,\gamma} + (\mathcal{L}_{\anchor(\tilde\alpha)}f)\gamma.$$
Both terms on the right hand side vanish at $y$: the first because $f$ does so and the second because $\anchor(\tilde\alpha)$ vanishes at $y$.

Since $\tilde\alpha, \tilde \beta \in  C^\infty(L)$ and $L$ is isotropic, we have $\Cour{\tilde{\alpha},\tilde{\beta}} = - \Cour{\tilde{\beta},\tilde\alpha}$. So the bracket is skew and the argument above also implies that the bracket is independent of the extension  $\tilde\alpha$.

Next we check that the right hand side lies in $N_y^{*1,0}$. Since the result is independent of the extensions we may pick $\tilde\alpha,\tilde\beta$ such that $\tilde\alpha|_Y,\tilde\beta|_Y \in C^\infty(N^{*1,0}Y)$ and we must check that $\langle\lb\tilde\alpha,\tilde\beta\rb,X\rangle$ vanishes for all $X \in C^\infty(TY)$. Indeed, we have
$$\langle\lb\tilde\alpha,\tilde\beta\rb,X\rangle = \mathcal{L}_{\anchor \tilde\alpha}\langle\tilde\beta,X\rangle - \langle\tilde\beta,\lb\tilde\alpha,X\rb\rangle.$$
The first term vanishes because $\anchor\tilde\alpha$ vanishes at $y$. The second vanishes because, over $Y$, $\anchor\tilde\alpha = 0$ and $X \in TY$, hence the vector part of $\Cour{\tilde\alpha,X}$ vanishes over $Y$ and $\tilde\beta\in C^\infty(N^*Y)$. 
\end{proof}

\begin{defn}\label{APB}
A generalized Poisson submanifold $i:Y \to (M,H,\J)$ is {\it Abelian}\/ if $N^{*1,0}Y\to Y$ is a bundle of Abelian Lie algebras.
\end{defn}

If $Y$ is a generalized Poisson submanifold with complex codimension one, then it is automatically Abelian as any one-dimensional Lie algebra is Abelian. From the Poisson viewpoint, the Abelian condition means that the induced Poisson structure vanishes quadratically in normal directions.

The last condition we introduce ensures that $Y$ itself carries a holomorphic structure.
\begin{defn}\label{GCB}
A \textsl{generalized complex brane} in $(M,H,\J)$ is a submanifold $Y\subset M$ together with a $\J$-invariant maximal isotropic subbundle $\tau \subset N^*Y^\perp$ such that $\anchor\colon\tau\to TY$ is surjective and the image of $\tau$ in $N^*Y^\perp/N^*Y$ is involutive.
\end{defn}
Given a brane $(Y,\tau)$, it follows that $N^*Y \subset \tau$ as by definition $N^*Y$ annihilates all elements in $N^*Y^\perp$ and $\tau$ is maximal. Dimension count shows that we have an exact sequence
\begin{equation}\label{eq:brane exact seq}
0 \to N^*Y \to \tau \to TY\to 0.
\end{equation} 
Using a splitting $\T M|_Y=TM|_Y\oplus T^\ast M|_Y$, the brane condition is equivalent to the existence of a two-form $F\in\Omega^2(Y)$ satisfying $dF=i^\ast H$, and for which 
\begin{align*}
\tau(F):=\{X+\xi\in TY\oplus T^\ast M| \ \xi|_{TY}= \iota_XF \} \subset \T M|_Y
\end{align*}
is $\J$-invariant.

If $\J$ is induced by a complex structure, a submanifold $i:Y\to M$ is a brane only if $Y$ is a complex submanifold. For symplectic manifolds,  branes include Lagrangian submanifolds.

The manifolds we are interested in are simultaneouly Abelian generalized Poisson submanifolds and branes. 
\begin{defn}\label{Poisson brane}
 An \textsl{(Abelian) Poisson brane} is an (Abelian) generalized Poisson submanifold that in addition carries the structure of a generalized complex brane.
\end{defn}

In the next few lemmas we make evident the holomorphic nature of Abelian Poisson branes.

The first lemma regards Poisson branes in general. Let $i:(Y,\tau) \to (M,H,\J)$ be a Poisson brane and $\J_Y$ be the induced generalized complex structure on $Y$.  Since $(Y,\tau)$ is a Poisson brane, the image of $\tau$ in $N^*Y^\perp/N^* Y$ provides an isotropic complement to $T^*Y$ which is involutive and invariant under the induced generalized complex structure. It follows from Example \ref{11:37:30c} that the generalized complex structure on $Y$ is given by a holomorphic Poisson structure. Since all data that we used to construct the holomorphic Poisson structure and corresponding complex structure was provided by the brane, these are naturally induced structures. We summarise this in the first lemma:

\begin{lem}\label{lem:pb=>hol poisson}
If $i:(Y,\tau) \to (M,H,\J)$ is a Poisson brane and $\J_Y$ is the induced generalized complex structure on $Y$ then $\J_Y$ is naturally equivalent to a holomorphic Poisson structure on $Y$. In particular $Y$ naturally inherits a complex structure.
\end{lem}

Since for a generalized Poisson submanifold $\J N^*Y = N^*Y$, the conormal bundle of $Y$ is a complex vector bundle over $Y$. In the case of an Abelian Poisson brane we can  further endow $N^*Y$ with a holomorphic structure.

\begin{lem}\label{lem:apb=>hol normal bundle}
Let $i:(Y,\tau) \to (M,H,\J)$ be an Abelian Poisson brane. Then $N^{*1,0}Y$ admits the structure of a generalized holomorphic bundle over $Y$ with the partial connection given by
$$\delbar_v \alpha = \lb\tilde v,\tilde \alpha\rb|_Y,$$
where $v\in C^\infty(L_Y)$, $\alpha \in C^\infty(N^{*1,0}Y)$ and $\tilde{v},\tilde \alpha$ are sections of $L$ such that $\tilde\alpha$ extends $\alpha$ and $\tilde{v}|_Y \in C^\infty((N^*Y^\perp)^{1,0})$ is a lift of $v$.

In particular, the partial connection above makes $N^{*1,0}$ into a holomorphic bundle for the underlying complex manifold $Y$.
\end{lem}
\begin{proof}
There is a number of things we need to check. First, we need to check that the proposed expression for the partial connection does not depend on the particular lift (and extension) $\tilde\alpha$ and $\tilde{v}$ of $\alpha$ and $v$. Then we must show that the proposed expression does indeed define a partial connection. 

Once a lift for $v$ is fixed, the proof that the expression does not depend on the extensions is similar to that of Lemma \ref{lem:lie algebra bundle} so we will omit it. To show that the partial connection is independent of the lift of $v$, we recall that the generalized complex structure on $Y$ is obtained from the quotient of complex vector bundles
$\T Y \cong N^*Y^\perp/N^*Y$, hence if $\tilde v$ is a lift of $v$ over $Y$, any other lift will differ from $\tilde{v}$ by a section $\tilde \beta \in C^\infty(N^{*1,0}Y)$ which must still be extended to a section of $L$. Therefore to show that the expression is independent of the lift we compute
$$ \lb\tilde v+\tilde\beta,\tilde \alpha\rb|_Y = \lb\tilde v,\tilde \alpha\rb|_Y + \lb\tilde\beta,\tilde\alpha\rb|_Y=\lb\tilde v,\tilde \alpha\rb|_Y,$$
where the term $\lb\tilde\beta,\tilde\alpha\rb|_Y$ vanishes because $N^{*1,0}Y$ is a bundle of Abelian Lie algebras. 

The fact that $\delbar$ defined this way is $C^\infty$-linear on $v$ follows from isotropy of $L$, property iii) from Lemma \ref{19:16:48} and the fact that $\alpha$ has no tangent component. Finally, the Jacobi identity for the Courant bracket implies that $\delbar^2=0$.

Since the induced generalized complex structure on $Y$ is holomorphic Poisson, a generalized holomorphic bundle is automatically holomorphic for the underlying complex structure.
\end{proof}

\section{Deformation theory of generalized complex structures}\label{14:20:32}

Next we study the deformation theory of generalized complex manifolds. We start,  in Section \ref{subsec:framework}, with the general framework of deformations of Dirac structures, introduced in \cite{MR2811595,MR1472888}. Eventually we narrow our focus to deformations of holomorphic Poisson structures and in Section \ref{sec:17Jan:0027} we prove the result on stability of  holomorphic Poisson structures on compact manifolds. The proof of this result helps us to set our strategy for the case of manifolds with boundary and a critical look at the proof signals where we can expect difficulties later.

\subsection{The framework for deformations}\label{subsec:framework}

Given a pair of complimentary almost Dirac structures, $L, \tL\subset \T M_\C$, we can identify $\tL \cong L^\ast$ using the natural pairing:
\begin{align*}
u(v):=2\langle u,v\rangle \hspace{10mm} \forall u\in \widetilde{L}, \ v\in L. 
\end{align*}
This allows us to define a de Rham differential $d_L:C^\infty(\Lambda^k \widetilde{L})\rightarrow C^\infty(\Lambda^{k+1} \widetilde{L})$ in the usual way by
\begin{equation}\label{eq:dl}
d_L\alpha(v_0,\ldots,v_k):=\sum_{i} (-1)^iv_i\cdot \alpha(\ldots, \widehat{v_i},\ldots ) +\sum_{i<j}(-1)^{i+j} \alpha(\lb v_i,v_j\rb^L,\ldots, \widehat{v_i},\ldots,\widehat{v_j},\ldots),
\end{equation}
where $\alpha\in C^\infty(\Lambda^k \widetilde{L})$, $v_0,\ldots,v_k\in C^\infty(L)$, and where $\lb v_i,v_j\rb^L$ denotes the component of $\lb v_i,v_j\rb$ in $L$ with respect to the decomposition $\T M_\C=L\oplus \widetilde{L}$. Although in general $d_L$ depends on the choice of $\tL$ we suppress this from the notation. 

If $L$ is integrable we have $d_L^2=0$, while if $\tL$ is integrable then the Courant bracket on $\tL$ extends to give a Lie bracket $\lb \cdot,\cdot \rb$ on $C^\infty(\Lambda^\ast \widetilde{L})$. 
If $L$ and $\tL$ are both integrable, which we will assume from now on, then the triple $(C^\infty(\Lambda^\bullet \widetilde{L}), \lb \cdot, \cdot \rb, d_L)$ constitutes a differential graded Lie algebra, where the grading on $C^\infty(\Lambda^\bullet \widetilde{L})$ is shifted by $1$ (see \cite{MR1472888}).

We can use $\tL$ to describe small deformations of $L$ as an almost Dirac structure. For $\ve\in C^\infty(\Lambda^2 \widetilde{L})$ we define another almost Dirac structure
\begin{align*}
L_\ve:=\{u+\iota_u\varepsilon | u\in L\},
\end{align*}
where $\iota_u\ve=\ve(u)\in\tL$ denotes the result of interior contraction. Note that every deformation of $L$ that is transverse to $\tL$ can be described in this way for a unique $\ve$, and this applies in particular to all small deformations of $L$. 
As $\tL$ remains complementary to $L_\ve$ we can again identify $\widetilde{L}\cong L_\ve^\ast$ for all $\ve$. This allows us to regard the corresponding de Rham operators as differential operators in a fixed vector bundle: $d_{L_\ve}:C^\infty(\Lambda^k \widetilde{L})\rightarrow C^\infty(\Lambda^{k+1} \widetilde{L})$. As shown in \cite{MR1472888}, the Dirac structure $L_\ve$ is integrable if and only if $\ve$ satisfies the \textsl{Maurer-Cartan equation} 
\begin{align}\label{13:47:45} 
d_L\varepsilon +\frac{1}{2}\lb \varepsilon,\varepsilon \rb=0.
\end{align} 
If this is the case then the corresponding deformed operator $d_{L_\ve}$ on $C^\infty(\Lambda^\bullet \widetilde{L})$ is given by 
\begin{align*}
d_{L_\varepsilon}=d_L+\lb \varepsilon, \cdot \rb.
\end{align*}

Having established that Maurer--Cartan elements in $ C^\infty(\Lambda^{2} \widetilde{L})$ describe nearby integrable structures, the next step in the study of deformations is to determine which deformations are equivalent to each other. That is, we need to describe how symmetries of $\T M$ act on deformations. If $F:\T M\rightarrow \T M$ is an automorphism which is sufficiently small (i.e.\ close to the identity) and if $\ve\in C^\infty(\Lambda^2\tL)$ describes a small deformation $L_\ve$, then $F(L_\ve)$ is another small deformation of $L$ that we can therefore write as $L_{F\cdot \ve}$ for a unique element $F\cdot \ve\in C^\infty (\Lambda^2\tL)$. In particular, for $u\in C^\infty(\T M)$ with flow $F_{tu}$ (see \eqref{10:12:54}) we can consider $F_{tu}\cdot \ve$ and differentiate it at $t=0$, inducing an infinitesimal action of $C^\infty(\T M)$ on $C^\infty (\Lambda^2\tL)$. Explicitly, one can show that 
\begin{align}\label{4:10:01}
\left. \frac{d}{dt}\right |_{t=0} F_{tu}\cdot \ve=d_{L_\ve} (u^{\tL}),
\end{align}
where $u=u^{L_\ve}+u^{\tL}$ is the decomposition of $u$ with respect to $\T M_\C=L_\ve\oplus \tL$.


\subsubsection*{Holomorphic Poisson structures}

The study of deformations becomes more concrete for certain special types of generalized complex manifolds. Here we focus on holomorphic Poisson structures.

Let $(M,I,\sigma)$ be a holomorphic Poisson manifold (with zero three-form on $M$) and let 
\begin{align*}
L:=T^{0,1}M\oplus e^\sigma(T^{\ast 1,0}M)=\{X+\sigma(\xi)+\xi|\ X\in T^{0,1}M, \ \xi \in T^{\ast 1,0}M\}
\end{align*}
be the corresponding Dirac structure. We choose $\tL:=T^{1,0}M\oplus T^{\ast 0,1}M$ as an integrable Dirac complement to $L$, for which one can compute that $d_L=\delbar+\lb \sigma, \cdot \rb$. 
A deformation $\ve\in C^\infty(\Lambda^2\tL)$ of $L$ is then integrable if and only if 
\begin{align}\label{09:10:55}
0=d_L\ve+\frac{1}{2}\lb \ve,\ve\rb=\delbar \ve+\lb \sigma,\ve\rb +\frac{1}{2}\lb \ve,\ve\rb. 
\end{align} 
Since $\widetilde{L}$ is a direct sum, its exterior algebra admits a natural splitting. It is useful to express the deformation parameter in terms of this splitting.  So we write $\ve=\ve_1+\ve_2+\ve_3$ with respect to the decomposition
\begin{align}\label{14:30:49}
\Lambda^2\tL = \Lambda^2 T^{1,0}M \oplus ( T^{1,0}M\otimes T^{\ast 0,1}M ) \oplus \Lambda^2 T^{\ast 0,1}M.									
\end{align}
A useful point of view is to interpret $\ve_2\in C^\infty(T^{1,0}M\otimes T^{\ast 0,1}M)$ as a (not necessarily integrable) deformation of the complex structure $I$. With this in mind we define the corresponding $\delbar$ operator:
\begin{align*}
\delbar_{\ve_2}:=\delbar + \lb \ve_2,\cdot\rb :\Omega^{0,q}(\Lambda^p T^{1,0}M)\rightarrow \Omega^{0,q+1}(\Lambda^p T^{1,0}M),
\end{align*}
where $\Omega^{0,q}(\Lambda^p T^{1,0}M):=C^\infty(\Lambda^{q}T^{\ast 0,1}M \otimes \Lambda^p T^{1,0}M)$. The Maurer-Cartan equation \eqref{09:10:55} can then be decomposed into four separate equations as follows:
\begin{alignat}{3}
              \Omega^{0,0}(\Lambda^3 T^{1,0}M): \hspace{15mm}& &&0=\lb \sigma,\ve_1\rb+\frac{1}2 \lb \ve_1,\ve_1\rb, \label{13:43:33} \\
    \Omega^{0,1}(\Lambda^2T^{1,0}M):\hspace{15mm}& &&0=\delbar_{\ve_2}\ve_1 + \lb \sigma,\ve_2 \rb, \label{13:43:53} \\
    \Omega^{0,2}(\Lambda^1T^{1,0}M):\hspace{15mm}& && 0= \delbar \ve_2 +\frac{1}{2}\lb \ve_2,\ve_2\rb+ \lb\sigma+\ve_1,\ve_3\rb,          \label{12:37:22} \\
 \Omega^{0,3}(\Lambda^0T^{1,0}M): \hspace{15mm}&  		&&0=\delbar_{\ve_2}\ve_3.   \label{1:07:34}
\end{alignat}
To gain some intuition for these equations, observe that the bracket $\lb\sigma+\ve_1,\ve_3\rb$ in \eqref{12:37:22} gives the obstruction for $I_{\ve_2}$, the deformation of $I$ with respect to $\ve_2$, to be integrable\footnote{Note that $I_{\ve_2}$ is integrable precisely when $\delbar_{\ve_2}^2=0$, or equivalently when $\delbar \ve_2+\frac{1}{2}\lb \ve_2,\ve_2\rb=0$. }. In the special case that this obstruction vanishes, the other three equations can be interpreted in clear terms: Equation \eqref{13:43:33} together with $\lb\sigma,\sigma\rb=0$ implies that $\sigma+\ve_1$ defines a Poisson structure. Equation \eqref{13:43:53} is equivalent to $\sigma+\ve_1$ being holomorphic for the complex structure $I_{\ve_2}$. Finally, \eqref{1:07:34} states that $\ve_3$ defines a holomorphic two-form for $I_{\ve_2}$. Note in particular that the deformed Dirac structure $L_\ve$ is given by a holomorphic Poisson structure if and only if $\ve_3=0$.

It is rare that a deformation of a holomorphic Poisson structure presents itself naturally already with vanishing component $\ve_3$. To achieve that, one needs to find a holomorphic gauge for it, that is we need to find an equivalent deformation using the infinitesimal action of one-forms on deformations as defined in \eqref{4:10:01} explicitly in this context. For $\xi\in\Omega^1(M)$, its decomposition with respect to $\T M_\C=L_\ve\oplus \tL$ is given by 
\begin{align*}
\xi=\big(\xi^{1,0}+(\sigma+\ve_1)(\xi^{1,0})+\ve_2(\xi^{1,0})\big)+\big(\xi^{0,1}-(\sigma+\ve_1)(\xi^{1,0})-\ve_2(\xi^{1,0})\big).
\end{align*}
In particular, if $F_{t\xi}=e^{td\xi}$ denotes the flow of $\xi$ (see \ref{10:12:54}), we deduce from (\ref{4:10:01}) that
\begin{align*} 
\left. \frac{d}{dt}\right |_{t=0} F_{t\xi}\cdot \ve =(\delbar+\lb \sigma+ \ve,\cdot \rb) \big(\xi^{0,1}-(\sigma+\ve_1)(\xi^{1,0})-\ve_2(\xi^{1,0})\big).
\end{align*}
We are mainly interested in the component that lies in $\Omega^{0,2}(M)$, which is given by
\begin{align}
\left. \frac{d}{dt}\right |_{t=0} \big(F_{t\xi}\cdot \ve \big)_3=& \ \delbar_{\ve_2}\big(\xi^{0,1}-\ve_2(\xi^{1,0})\big)-\lb \ve_3,(\sigma+\ve_1)(\xi^{1,0})\rb. \label{13:21:37}
\end{align}

It is actually also possible to describe the action of closed two-forms.
\begin{lem}\label{17:03:57} Let $B$ be a sufficiently small closed real two-form. Then
\begin{align}
(e^B\cdot \ve)_1=& (\ve_1-\sigma B^{2,0}(\sigma+\ve_1))(1+B^{2,0}(\sigma+\ve_1))^{-1}\label{09:31:48}\\
(e^B\cdot \ve)_2=& (\ve_2+B^{1,1}(\sigma+\ve_1))(1+B^{2,0}(\sigma+\ve_1))^{-1}\\
\label{17:36:15} 	(e^B\cdot \ve)_3=& \ve_3+B^{0,2}+B^{1,1}\ve_2-(\ve_2+B^{1,1}(\sigma+\ve_1))(1+B^{2,0}(\sigma+\ve_1))^{-1}(B^{1,1}+B^{2,0}\ve_2) 
\end{align} 
\end{lem}
\begin{proof}
By definition we have $e^B(L_\ve)=L_{e^B\cdot \ve}$, where $L=T^{0,1}M\oplus T^{\ast 1,0}M$. Writing out both sides, a tedious yet straightforward calculation, yields the above three equations. 
\end{proof}
\begin{rem}\label{11:43:33}
Lemma \ref{17:03:57} treats the various components of $\ve$ and $B$ as (skew-symmetric) endomorphisms. For instance, $\ve_2\in C^\infty(T^{1,0}M\otimes T^{\ast 0,1}M)$ is considered as a map $\ve_2:T^{\ast 1,0}M\rightarrow T^{\ast 0,1}M$, the $(2,0)$-part of $B$ as a map $B^{2,0}:T^{ 1,0}M\rightarrow T^{\ast 1,0}M$, and so on. From these expressions we see that $B$ being small in the hypothesis of Lemma \ref{17:03:57} amounts to the endomorphism $1+B^{2,0}(\sigma+\ve_1)$ being invertible, for which a sufficient condition is given by
\begin{align}\label{13:58:55}
|B|_0 < \frac{1}{|\sigma+\ve_1|_0}.
\end{align}
\end{rem}

\begin{rem}\label{20:13:41} 
Equation \eqref{13:21:37} can be obtained from \eqref{17:36:15} by taking $B=td\xi$ and differentiating at $t=0$. Specifically, this yields the equality 
\begin{align}\label{12:30:54}
d\xi^{0,2}+d\xi^{1,1}\ve_2-\ve_2(d\xi^{1,1}+d\xi^{2,0}\ve_2)=\delbar_{\ve_2}\big(\xi^{0,1}-\ve_2(\xi^{1,0})\big)-\lb \ve_3,(\sigma+\ve_1)(\xi^{1,0})\rb.
\end{align}  
\end{rem}

\begin{rem}\label{rem:24-08-22}
Notice that even though our problem is directly related to deformations its setup is different from the standard type of deformation problem. Indeed, for standard deformation problems one has an {\it integrability} condition which is {\it preserved by the action} of a group of symmetries. In our case, we assume from the start that we are dealing with integrable generalized complex structures and the group of symmetries (the $B$-field transforms) does not preserve ``being in holomorphic Poisson gauge''. This is precisely what allows us to use symmetries to obtain the main results. A consequence of this key difference is that there are no ready-made deformation theory results we can use to arrive at the desired conclusion. Instead we have to work with the underlying analytical results and adapt them to the present situation.
\end{rem}

\subsection{Stability of holomorphic Poisson structures on compact manifolds}\label{sec:17Jan:0027}

As we mentioned last section, a deformation of a holomorphic Poisson structure is itself manifestly holomorphic Poisson if the component $\ve_3 \in C^\infty(\wedge^2 T^{*0,1}M)$ of the deformation parameter vanishes. Our aim, throughout this paper, is to establish conditions under which this component can be made to vanish by the action of symmetries of $\T M$.

In this section we establish when this is the case for compact manifolds without boundary. The main result of this section depends on the behaviour of the map $H^2(M;\R)\rightarrow H^{0,2}(M,I)$ which assigns to any real two-form its $(0,2)$-part:

\begin{thm}\label{15:13:30}
Let $(M^{2n},I,\sigma)$ be a compact holomorphic Poisson manifold without boundary such that $H^2(M;\R)\rightarrow H^{0,2}(M,I)$ is surjective. Then every  generalized complex deformation of $(I,\sigma)$ sufficiently small in the Sobolev $L^2_{n+3}$-norm  admits a holomorphic gauge, that is, it is $B$-field equivalent to a holomorphic Poisson deformation of $(I,\sigma)$. If in addition $H^{0,2}(M)=0$, then the two-form $B$ may be taken to be exact.
\end{thm}

This theorem implies, for example, that if $H^{0,2}(M) = 0$, then every small deformation of a holomorphic Poisson structure is again (equivalent to) holomorphic Poisson, albeit for a (possibly) different complex structure on $M$. That is if $H^{0,2}(M) = 0$ `being holomorphic Poisson' is an open condition in generalized complex geometry.

\begin{proof} 
Fix a Hermitian metric on $(M,I)$, let $\Delta$ denote the corresponding $\delbar$-Laplacian and $\mathcal{H}^{p,q}$ the space of $\delbar$-harmonic $(p,q)$-forms. Let $G$ be the associated Green's operator and let $H$ be the projection onto harmonics, so that $\Id=\Delta G+H$. Let $s:\mathcal{H}^{0,2} \to \Omega^2_{cl}(M;\R)$ be a right inverse for the projection $p\colon \Omega^2_{cl}(M;\R)\rightarrow H^{0,2}(M,I)\cong \mathcal{H}^{0,2}$, where $\Omega^2_{cl}(M;\R)$ is the space of closed real $2$-forms, which we may arrange to satisfy\footnote{Note that any splitting $\tilde{s}$ satisfies $B=\tilde{s}(B)^{0,2}+\delbar \alpha$ for some $(0,1)$-form $\alpha$. We can then change $\tilde{s}$ (using a finite basis for $\mathcal{H}^{0,2}$) via $s(B):=\tilde{s}(B)+d(\alpha+\overline{\alpha})$.} $s(B)^{0,2}=B$. 

Denote by $L^2_k(\text{Im}(\delbar^\ast))$ the image of $\delbar^\ast\colon L^2_{k+1}(\Lambda^{0,2}T^\ast M)\rightarrow L^2_k(\Lambda^{0,1}T^\ast M)$, and consider the smooth maps:
\[
\beta \colon L^2_k(\Lambda^2\tL)\oplus\mathcal{H}^{0,2} \oplus L^2_k(\text{Im}(\delbar^\ast))  \to L^2_{k-1}(\wedge^{0,2}T^*M)
\]
\[\beta(\ve,B,u)  =(e^{s(B)+d(u+\bar{u})}\cdot \ve)_3
\]
and
\begin{align}\label{PHI}
\Phi: L^2_k(\Lambda^2\tL)\oplus \mathcal{H}^{0,2} \oplus L^2_k(\text{Im}(\delbar^\ast))  \rightarrow \mathcal{H}^{0,2} \oplus L^2_k(\text{Im}(\delbar^\ast)) 
\end{align}
\[
\Phi(\ve,B, u)=  (H(\beta(\ve,B,u)),\delbar^\ast G (\beta(\ve,B,u))).
 \]
Here $k:=n+3$ is chosen to ensure that the Sobolev spaces of degree $k$ are algebras and that their elements are once continuously differentiable (the latter will be relevant below).
Denote by $D^{(2,3)}_{(\ve,B,u)}\Phi$ the partial derivative of $\Phi$ with respect to its second and third variable at the point $(\ve,B,u)$, which is a linear map
\begin{align*}
D^{(2,3)}_{(\ve,B,u)}\Phi: \mathcal{H}^{0,2} \oplus L^2_k(\text{Im}(\delbar^\ast)) & \rightarrow \mathcal{H}^{0,2} \oplus L^2_k(\text{Im}(\delbar^\ast)).
 \end{align*}
We claim that $D^{(2,3)}_{(0,0,0)}\Phi$ is invertible. Indeed, due to \eqref{13:21:37} and \eqref{17:36:15},
\begin{align*}
D_{(0,0,0)}\beta(0,B,u)=
s(B)^{0,2}+\delbar u=B+\delbar u
\end{align*} 
by our choice of splitting $s$, hence
\begin{align*}
D^{(2,3)}_{(0,0,0)}\Phi(B,u)=(H(B+\delbar u),\delbar^\ast G (B+\delbar u)) = (B, u).
\end{align*}
Here we used that $\delbar^\ast G \delbar u = u$ because $u \in \text{Im}(\delbar^\ast)$. By the implicit function theorem we obtain, for every small deformation $\ve \in L^2_k(\wedge^2 \tL)$, a pair $(B,u) \in  \mathcal{H}^{0,2} \oplus L^2_k(\text{Im}(\delbar^\ast)) $ with the property that after applying the closed two-form $s(B)+ d(u+ \bar{u})$ to $\ve$ we have $(H(\ve_3),\delbar^\ast G(\ve_3))=(0,0)$. The pair $(B,u)$ depends smoothly on $\ve$, so by taking $\ve$ $L^2_k$-small we can ensure that $(B,u)$ is sufficiently $C^1$-small (because $k>n+2$).

To finish the proof we must prove two facts:
\begin{enumerate}
\item  If $(H(\ve_3),\delbar^\ast G(\ve_3))=(0,0)$ and $\ve$ is small enough and integrable, then $\ve_3=0$ and therefore the deformation is holomorphic Poisson.
\item If $\ve \in L^2_k(\Lambda^2\tL)$ is small enough and in fact smooth, then the pair $(B,u)$ obtained above is also smooth. 
\end{enumerate}

To prove 1.\ we observe that if $(H(\ve_3),\delbar^\ast G(\ve_3))=(0,0)$, then 
\begin{align*}
0=\delbar\delbar^\ast G\ve_3=\Delta G\ve_3-\delbar^\ast G \delbar\ve_3=\ve_3+\delbar^\ast G \lb \ve_2,\ve_3\rb.
\end{align*}
In the last step we used that $H(\ve_3)=0$ and that $\ve$ satisfies the Maurer-Cartan equation. The operator 
\begin{align*}
1+\delbar^\ast G \lb \ve_2,\cdot\rb : L^2_k(\Lambda^{0,2}T^\ast M)\rightarrow L^2_k(\Lambda^{0,2}T^\ast M)
\end{align*}
is invertible for small $\ve_2$, which proves 1.

To prove 2.\ we observe that by construction $s(B)$ is smooth, so the only question here is whether the (0,1)-form $u$ chosen above is smooth. To prove that this is the case, we turn to \eqref{17:36:15}. If we let $F_1$ be the nonlinear smooth bundle map
\[
F_1\colon \wedge^2(T^{1,0}M\oplus T^{*0,1}M)\times \wedge^2T^*M \to \wedge^{0,2} T^*M,
\]
\[
F_1(\ve, B)= \ve_3+B^{0,2}+B^{1,1}\ve_2-(\ve_2+B^{1,1}(\sigma+\ve_1))(1+B^{2,0}(\sigma+\ve_1))^{-1}(B^{1,1}+B^{2,0}\ve_2),
\]
then it follows from 1.\ that $F_1(\ve, s(B) + d(u + \bar{u})) =0$. Therefore, not only $u \in \operatorname{Im}(\delbar^*)$ but it is also a solution to a nonlinear differential equation of first order.  Using these two facts we conclude that $u$ is also a solution to the second order differential equation
\begin{equation}\label{eq:quasilinear}
F(\ve,B,u) := \delbar^* F_1(\ve, s(B) + d(u + \bar{u})) + \delbar \delbar ^* u =0,
\end{equation}
since both summands vanish independently. The first summand is a second order quasilinear differential operator of divergence type in the sense of \cite{MR0244627}, since it is a composition of a nonlinear first order differential operator with $\delbar^*$ (a first order {\it linear} differential operator).  The second summand is clearly a linear differential operator. Therefore $F(\ve,B,\cdot)$ is also a second order quasilinear differential operator of divergence type.  Since $F_1$, $\ve$ and $s(B)$ are smooth, according to \cite[Theorem 6.3]{MR0244627}, a solution $u$ of this equation is smooth as long as the linearization of $F(\ve,B,\cdot)$ at $u$  is an elliptic (linear) differential operator. Therefore all that is left to finish the proof is to argue that this is the case. Since $F$ is a quasi-linear operator the principal symbol of its linearization at $u$ only depends on the first jet of $u$. Further, as mentioned before, as long as  $\ve$ is chosen small enough in $L^2_k(\wedge^{2,0}T^*M)$, we obtain that $B$  and $u$ are $C^1$-small, hence, by continuity of the principal symbol, to show that the linearization of $F(\ve,B,\cdot)$ is elliptic at $u$ it suffices to show it is elliptic for vanishing $\ve$, $B$, $u$. For $\ve = B =0$, we have
\[
F(0,0,u) = \delbar^*(\delbar u-((\del u + \delbar \bar u)\sigma)(1+\del \bar u \sigma)^{-1}(\del u + \delbar \bar u)) + \delbar \delbar^* u,
\]
whose linearization at $u\equiv 0$ is $\frac{\del F}{\del u}|_{(0,0,0)} (v) = \triangle v$ which is elliptic, hence the result follows.
\end{proof}

One of our main results in this paper is that under a few additional hypothesis the theorem above also holds for manifolds with boundary. At this stage it is good to take a critical look at the proof of Theorem \ref{15:13:30} to prepare ourselves for what is to come. The broad idea of our approach remains to use small symmetries of $\T M$ to kill the component $\ve_3$ of a given small deformation. Above,  we used two key analytical results to achieve that: Hodge theory for the $\delbar$-operator and the implicit function theorem for Banach spaces.

Hodge theory for complex manifolds with boundary was developed by Folland and Kohn in \cite{MR0461588}. While it is similar to usual Hodge theory, there is a subtle, but fundamental difference: the inverse of the Laplacian, now called $N$, the {\it Neumann operator}, only recovers one weak derivative. If we follow the approach above with this less sharp Neumann operator, this apparently technical difference comes back at the map $\Phi$ defined in \eqref{PHI}, which in this new context has a different codomain:
\[
\Phi: L^2_k(\Lambda^2\tL)\oplus \mathcal{H}^{0,2} \oplus L^2_{k}(\text{Im}(\delbar^\ast))  \rightarrow \mathcal{H}^{0,2} \oplus L^2_{k-1}(\text{Im}(\delbar^\ast)). 
\]
 The same computation carried above shows that the derivative of $\Phi$ along the second and third coordinates at the identity corresponds to the natural inclusion 
\[
\mathcal{H}^{0,2} \oplus L^2_{k}(\text{Im}(\delbar^\ast))  \to \mathcal{H}^{0,2} \oplus L^2_{k-1}(\text{Im}(\delbar^\ast))
\]
 which does not allow us to apply the implicit function theorem.
%
 
This small and apparently technical issue irreparably breaks down the argument used in Theorem \ref{15:13:30} and we must instead find another approach that can deal with this ``loss of derivatives''. This is typically the context of Nash--Moser types of arguments \cite{MR3128977,MR656198,MR2855089,MR3261013} and this is precisely  the path that will allow us to prove the final result. 

To use a Nash--Moser type of algorithm it is not enough to have Hodge theory developed for one fixed structure. We need uniform bounds on the Neumann operator for all nearby structures, integrable or not. Extending the work of Folland and Kohn  to keep track of how the Neumann operators of nearby structures are related is an undertaking of its own, and was carried out by van der Leer Dur\'an separately in \cite{s00208-021-02293-5}. We recall the main result of \cite{s00208-021-02293-5} in Section~\ref{sec:17Jan:0122}. With the Hodge theory for families of almost complex structures at hand, we can, in this paper, produce the appropriate Nash--Moser type of algorithm that will substitute the implicit function theorem.

Therefore the next few steps will be to set up Hodge theory for families of almost complex structures, which will allow us to give a precise statement of our main theorem. After that we will set up the Nash--Moser algorithm.
 
 \section{Hodge Theory on complex manifolds with boundary}\label{sec:hodge}
 Hodge theory for (almost) complex structures on manifolds with boundary depends on how convex the boundary is in a very precise way. In this section we will introduce $q$-covexity, the geometric notion necessary to develop Hodge theory for manifolds with boundary, and eventually state the result on elliptic regularity for the $\delbar$-Laplacian on complex manifolds with boundary (c.f. Theorem \ref{15:35:57}).
 
\subsection{Cauchy-Riemann structures and $q$-convexity}

Let $M$ be a manifold with boundary of real dimension $2n$ and let $I$ be an almost complex structure on $M$. On the boundary $\del M$ we can define a complex vector bundle 
 \begin{align*}
\mathcal{C}\mathcal{R}:=T^{0,1}M\cap T\del M_\C
 \end{align*}
called the \textsl{Cauchy-Riemann structure} of $(M,I)$, of complex rank $n-1$. 
Consider the complex line bundle $\mu$ on $\del M$ given by 
\begin{align*}
\mu:=\frac{T\del M_\C}{\mathcal{C}\mathcal{R}\oplus \overline{\mathcal{C}\mathcal{R}}}.
 \end{align*}
It is the complexification of a real line bundle that we can equip with a natural orientation as follows. For $x\in \del M$, let $\nu\in T_xM$ denote an outward pointing vector with the property that $I\nu\in T_x\del M$. Then $I\nu$ defines a real generator for $\mu_x$ which, up to a positive rescaling, is independent of the choice of $\nu$. 
\begin{defn}
The \textsl{Levi-form} of $(M,I)$ at $x$ is the Hermitian bilinear form defined on $\mathcal{C}\mathcal{R}_x$ with values in $\mu_{x}$ given by
\begin{align}\label{12:20:40}
\mathcal{L}_{x}(u,v):=-i[u,\overline{v}].
\end{align}  
\end{defn}
\begin{rem}
We are abusing notation here by implicitly choosing local extensions of $u$ and $v$ to sections of $\mathcal{C}\mathcal{R}$, in order to compute the bracket $[u,\overline{v}]$. By definition of $\mu$ as a quotient, this bracket is independent of the choice of extensions. 
\end{rem}
Using the identifications $\mu_x\cong \C$ described earlier, we can identify $\mcL$ with a conformal class of Hermitian bilinear forms on $\mathcal{C}\mathcal{R}$, so in particular we can consider the number of positive and negative eigenvalues of $\mcL$ at every boundary point.  

 \begin{defn}\label{11:30:04}
Let $(M,I)$ be an almost complex manifold of real dimension $2n$ and let $0\leq q\leq n$ be an integer. We call $(M,I)$ \textsl{$q$-convex}, if for every $x\in \del M$ the Levi-form $\mcL_x$ has either at least $n-q$ positive or at least $q+1$ negative eigenvalues at $x$. 
 \end{defn}
\begin{rem}
Note that $(M,I)$ is always $n$-convex. The terminology $q$-convex is not standard, and perhaps slightly misleading as negative eigenvalues correspond to concave boundaries. In \cite{MR0461588} (in the context of complex structures) $q$-convexity is referred to as \lq\lq condition $Z(q)$\rq\rq.
\end{rem}

\begin{ex} Let $(M',I)$ be an integrable complex manifold without boundary and let $M\subset M'$ be a submanifold with boundary. Around any point $x\in \del M$ we can find holomorphic coordinates $(z_1,\ldots,z_n)$ and a real smooth function $\varphi$ satisfying $M=\{y\in M'| \ \varphi(y)<0\}$ and $d_x\varphi \neq 0$. Then 
$$u=\sum_{i=1}^n u^i \del_{\overline{z}^i}\in T^{0,1}M$$ 
lies in $ \mcC\mathcal{R}$ if and only if $\sum_i u^i  \del_{\bar{z}_i} \varphi=0$, and on such vectors the Levi form is represented by the Hermitian bilinear form
\begin{align*}
\mcL(u,v)=\sum_{i,j=1}^n \frac{\del^2\varphi}{\del z^i\del \overline{z}^j} \overline{v^i} u^j.
\end{align*} 
If the above Hessian of $\varphi$ is positive definite on $T^{0,1}M_\C$ then its restriction to $\mcC\mathcal{R}$ is too and hence $(M,I)$ is $q$-convex for all $q\geq 1$. For instance, $\varphi:=|z|^2-1$ on the unit ball in $\C^n$ has this property, which is therefore $q$-convex for all $1\leq q\leq n$. If we then remove a smaller ball from its interior, the boundary has two components, on which $\mcL$ is positive and negative definite. This annular region is therefore $q$-convex for all $1\leq q\leq n-2$. 
\end{ex}   
\begin{ex}\label{10:11:15} 
Let $Y$ be an $(n-1)$-dimensional compact complex manifold without boundary and $p:L\rightarrow Y$ a holomorphic line bundle. Let $h$ be a Hermitian metric on $L$ and let $M\subset L$ be the associated unit disc-bundle. 
On $\del M$ we have $\mcC\mathcal{R}\cong p^\ast T^{0,1}Y$, and if $R_h$ denotes the curvature associated to $h$ then it turns out that 
 \begin{align*}
\hspace{15mm} \mcL(u,v)=R_h(u,\overline{v})\hspace{15mm} \forall u,v\in T^{0,1}Y.
\end{align*}
Consequently, $M$ is $q$-convex if and only if $-iR_h$ has, at each point $y\in Y$, either at least $n-q$ positive eigenvalues or at least $q+1$ negative eigenvalues\footnote{By definition, the eigenvalues of a real $(1,1)$-form $\tau$ are the eigenvalues of the Hermitian matrix $\tau_{ij}$ with respect to a decomposition $\tau=i\sum_{i,j}\tau_{ij}dz^i\wedge d\bar{z}^j$ in local coordinates.}. 
\end{ex}
Using the fact that $c_1(L)$ is represented by $\tfrac{i}{2\pi}R_h$, one can sometimes translate the previous example into a statement about $c_1(L)$.
\begin{lem}\cite{MR0137127}\label{11:16:39}
Let $Y$ satisfy the $\del\delbar$-lemma\footnote{Specifically, we need that every real $(1,1)$-form which is $d$-exact is also $\del\delbar$-exact.}. Then there exists a $q$-convex disc neighbourhood 
of $Y$ inside a holomorphic line bundle $L$  if and only if $c_1(L)$ has a real representative $\tau\in \Omega^{1,1}(Y)$ which at each point has either at least $n-q$ negative eigenvalues or at least $q+1$ positive eigenvalues.
\end{lem}
\begin{ex}\label{11:18:37}
If $Y$ is compact and $L$ is negative, meaning that $c_1(L)$ admits a representative whose eigenvalues are all negative, then $Y$ is K\"ahler and the $\del\delbar$-lemma holds. Consequently, there exists a disc-neighbourhood which is $q$-convex for all $q\geq 1$. The line bundles $\mcO_{\mathbb{P}^m}(-n)$ over $\mathbb{P}^m$ for $n>0$ satisfy this for example.
\end{ex}

\subsection{Relation between $q$-convexity and Hodge decompositions}\label{sec:17Jan:0122}

Let $M$ be a compact manifold with boundary and let $I$ be an integrable complex structure on $M$. The associated $\delbar$ operator defines a linear differential operator
$$\delbar: C^\infty(\Lambda^\bullet T^{\ast 0,1}M) \rightarrow C^\infty(\Lambda^{\bullet+1} T^{\ast 0,1}M),$$
which is a derivation in the sense that $\delbar(\alpha\wedge\beta)=\delbar\alpha\wedge\beta+(-1)^{\operatorname{deg}(\alpha)}\alpha\wedge\delbar\beta$, and satisfies $\delbar^2=0$. If $\ve_2\in \Omega^{0,1}(T^{1,0}M)$ describes an almost complex deformation of $I$, we can consider the differential operator  
$$\delbar_{\ve_2}=\delbar+\lb \ve_2,\cdot\rb: C^\infty(\Lambda^\bullet T^{\ast 0,1}M) \rightarrow C^\infty(\Lambda^{\bullet+1} T^{\ast 0,1}M).$$
It satisfies the same properties as $\delbar$, except that $\delbar_{\ve_2}^2$ vanishes if and only if $\ve_2$ defines an integrable deformation of $I$. 
\begin{rem}
For a more concrete expression of $\delbar_{\ve_2}$ we can write $\ve_2=\sum_j \alpha_j\otimes X_j$ where $\alpha_j\in T^{\ast 0,1}M$ and $X_j\in T^{1,0}M$. Then for any $\beta\in \Omega^{0,\bullet}(M)$ we have
\begin{align*}
\delbar_{\ve_2} \beta=\delbar \beta+\sum_j \alpha_j\wedge \iota_{X_j}d\beta. 
\end{align*}  
\end{rem}
Our aim here is to describe a Hodge decomposition for $\delbar_{\ve_2}$. Fix an auxiliary Hermitian metric on $(M,I)$ and let $\delbar_{\ve_2}^\ast$ be the corresponding formal adjoint of $\delbar_{\ve_2}$. Denote by
$$\Delta_{\ve_2}:=\delbar_{\ve_2}\delbar_{\ve_2}^\ast+\delbar_{\ve_2}^\ast\delbar_{\ve_2}$$
the associated Laplacian, which we consider as an unbounded operator on the Hilbert space $L^2(\Lambda^q T^{\ast 0,1}M)$ of square-integrable $(0,q)$-forms (here $q$ is an integer between $0$ and $n=\text{dim}_\C(M)$). The following proposition is due to Friedrichs \cite{MR1512905} (see also \cite{MR0461588} or \cite{s00208-021-02293-5} for more details). Below we will abbreviate $\Omega^{0,q}(M)=C^\infty(\Lambda^q T^{\ast 0,1}M)$ and denote by $r\in C^\infty(M)$ a function which is negative on the interior of $M$, zero on the boundary $\del M$ and satisfies $|dr|_{\del M}=1$ (with respect to the given Hermitian metric). We denote by $\sigma(\delbar_{\ve_2}^\ast,dr)$ the symbol of $\delbar_{\ve_2}^\ast$ evaluated on the one-form $dr$.
\begin{prop}[\cite{MR1512905}]\label{09:59:56}
There exists a self-adjoint unbounded operator\/ $\square_{\ve_2}$ on $L^2(\Lambda^q T^{\ast 0,1}M)$ whose domain satisfies
\begin{align}\label{08:58:09}
\text{Dom}(\square_{\ve_2})\cap \Omega^{0,q}(M)=\{\varphi\in \Omega^{0,q}(M)|\ \sigma(\delbar_{\ve_2}^\ast,dr)\varphi|_{\del M}=0,\ \sigma(\delbar_{\ve_2}^\ast,dr)\delbar_{\ve_2}\varphi|_{\del M}=0		\},
\end{align}
on which $\square_{\ve_2}$ agrees with with $\Delta_{\ve_2}$.
\end{prop}
The following theorem provides Hodge decompositions for the operator $\square_{\ve_2}$. Throughout this paper we'll denote by $||\cdot||_k$ the $L^2_k$-norm and by $|\cdot|_k$ the $C^k$-norm, abbreviating $||\cdot||:=||\cdot||_0$ and $|\cdot|:=|\cdot|_0$.  
\begin{thm}[\cite{s00208-021-02293-5}]\label{15:35:57}
Let $(M,I)$ be a compact complex manifold with boundary which is $q$-convex (see Def.\ref{11:30:04}). Then there exists a neighbourhood $B\subset \Omega^{0,1}(T^{1,0}M)$ of zero and an integer $b\in \mathbb{Z}_{\geq 0}$ such that the following hold for every $\ve_2\in B$:
\begin{itemize}
\item[1) ] 
If $\varphi\in \text{Dom}(\square_{\ve_2})$ has the property that $\square_{\ve_2}\varphi$ is smooth, then $\varphi$ is smooth as well and  \begin{align}\label{21:18:00}
||\varphi||_{k+1}\leq \mcL(|\ve_2|_{k+b};||\square_{\ve_2}\varphi||_k)+\mcL(|\ve_2|_{k+b};||\varphi||) 
\end{align}
for all $\varphi\in \text{Dom}(\square_{\ve_2})\cap \Omega^{0,q}(M)$ and all $k\in \mathbb{Z}_{\geq 0}$.
\item[2)] The image of $\square_{\ve_2}$ is closed and we have an orthogonal decomposition 
\begin{align}\label{16:38:18}
L^2(\Lambda^q T^{\ast 0,1}M)=\text{Im}(\square_{\ve_2})\oplus \mcH_{\ve_2}, 
\end{align}  
with $\mcH_{\ve_2}:=\text{Ker}(\square_{\ve_2})\subset \Omega^{0,q}(M)$ finite dimensional. 
\end{itemize}
The decomposition (\ref{16:38:18}) gives rise to the Neumann operator $N_{\ve_2}$ on $L^2(\Lambda^qT^{\ast 0,1}M)$, which by definition is zero on $\mcH_{\ve_2}$ and satisfies $N_{\ve_2} \square_{\ve_2}\varphi=(1-\pi_{\ve_2})\varphi$ for $\varphi\in \text{Dom}(\square_{\ve_2})$, where $\pi_{\ve_2}$ denotes the projection to $\mcH_{\ve_2}$. 
\begin{itemize}
\item[3)] For every fixed $\ve_2$ the operator $N_{\ve_2}$ is bounded, self-adjoint and induces bounded operators $N_{\ve_2}:L^2_k(\Lambda^qT^{\ast0,1}M)\rightarrow L^2_{k+1}(\Lambda^qT^{\ast0,1}M)$. For every $\varphi\in L^2(\Lambda^qT^{\ast0,1}M)$ we have $\varphi=\square_{\ve_2} N_{\ve_2}\varphi+\pi_{\ve_2}\varphi$, and if $\varphi\in \text{Dom}(\square_{\ve_2})$ then also $\varphi= N_{\ve_2}\square_{\ve_2}\varphi+\pi_{\ve_2}\varphi$.
\item[4)] If $\mcH_0=0$ then, after possibly shrinking $B$, we have $\mcH_{\ve_2}=0$ for all ${\ve_2}\in B$ and 
\begin{align}\label{15:53:51}
|N_{\ve_2}\varphi |_k\leq \mcL(|\ve_2|_{k+b};| \varphi |_{k+b}),
\end{align}
holds for all $\varphi \in \Omega^{0,q}(M)$ and all $k\in \mathbb{Z}_{\geq 0}$. 
\end{itemize}
\end{thm}
We will make precise what inequalities of the form \eqref{21:18:00} or \eqref{15:53:51} mean in Section \ref{sec:leibnizbounds} where we introduce Leibniz bounds and give some simple examples of those. For now it is enough to understand that these inequalities describe bounds on the $L^2_{k+1}$ and $C^{k+1}$ norms of the left hand side in terms of bounds on the corresponding norms of the right hand side, which includes information about the size of the deformation parameter. That is, these inequalities provide uniform bounds for the Neumann operator of nearby complex structures.

\begin{rem}
By this theorem, if $M$ is $q$-convex and $\varphi\in \Omega^{0,q}(M)$ then $\varphi=\Delta_{\ve_2} N_{\ve_2}\varphi+\pi_{\ve_2}\varphi$. However, only if $\varphi$ satisfies the two Neumann boundary conditions in \eqref{08:58:09} do we have 
\begin{align}\label{14:00:04}
\varphi= N_{\ve_2}\Delta_{\ve_2}\varphi+\pi_{\ve_2}\varphi.
\end{align} 
\end{rem}


\section{Main results}\label{sec:main results}

With Hodge theory under our belts we have nearly all the analytical tools to prove our main theorems. Since the proof is rather technical, we will first state our main results in this section, discuss some applications in Section \ref{Applications} and provide the proof  in the final section of this paper.
\subsection{Statement of the main theorems}

Let $M$ be a compact manifold with boundary, $I$ be an integrable complex structure on $M$ and $\sigma$ be a holomorphic Poisson structure on $(M,I)$. Our first main result concerns small deformations of the generalized complex structure that corresponds to $(I,\sigma)$. 

\begin{thm}\label{14:11:52}
Let $(M,I,\sigma)$ be a compact holomorphic Poisson manifold with boundary. Suppose that $(M,I)$ is $2$- and $3$-convex and that $H^{0,2}(M,I)=0$. Then any sufficiently small generalized complex deformation of $(M,I,\sigma)$ is $B$-field equivalent to another holomorphic Poisson structure on $M$ for some exact two-form $B$ on $M$.  
\end{thm}

\begin{rem}
Note how here, in contrast with Theorem \ref{15:13:30}, we require $H^{0,2}(M,I)=0$ instead of the weaker assumption that $H^2(M) \rightarrow H^{0,2}(M,I)$ is surjective. One of the main reasons we impose the stronger assumption here is that in the proof we heavily rely on the Hodge theory of Theorem \ref{15:35:57} for nearby deformations of $I$, and part 4) of that theorem requires the harmonics of $I$ to vanish to guarantee uniform bounds on the corresponding family of Neumann operators. 
\end{rem}

Our second main result concerns neighbourhoods of submanifolds in generalized complex geometry. Specifically, let $(M,H,\J)$ be a generalized complex manifold and $i:Y\hookrightarrow M$ a compact submanifold with the structure of an Abelian Poisson brane on it (see Def.\ref{Poisson brane}). Here $M$ and $Y$ are manifolds without boundary. From Lemmas \ref{lem:pb=>hol poisson} and \ref{lem:apb=>hol normal bundle} we know that $Y$ is equipped with a complex structure $I_Y$ and that the normal bundle $NY$ is a holomorphic vector bundle over $(Y,I_Y)$. The total space of $NY$ therefore comes equipped with an integrable complex structure denoted by $I$, which depends only on the brane-structure of $Y$.  

\begin{thm}\label{16:02:38}
Let $(M,H,\J)$ be a generalized complex manifold and $(Y,\tau)\subset M$ a compact Abelian Poisson brane. 
Denote by $I$ the induced complex structure on the total space of $NY$. If there exists a neighbourhood $U$ of $Y$ in $NY$ so that $(U,I)$ is $2$- and $3$-convex and $H^{0,2}(U,I)=0$, then there exists a neighbourhood $V$ of $Y$ in $M$ and an exact two-form $B\in\Omega^2(V)$ so that $\J$ is $B$-field equivalent to a holomorphic Poisson structure on $V$.     
\end{thm}

\begin{rem}
In the theorem above, the neighbourhood $U$ of $Y$ is assumed to be a compact submanifold of $NY$ with boundary, having zero codimension in $NY$ and containing $Y$ in its interior. This assumption will be implicit in the remainder of this paper when we consider a neighbourhood of a submanifold with appropriate convexity assumptions. In our applications, $U$ will always be the unit disc bundle with respect to some Hermitian metric.
\end{rem}

Even though this theorem is inspired by Theorem \ref{theo:Bailey}, which describes generalized complex structures around points, the general case of $Y$ being a point is not covered by Theorem \ref{16:02:38} because of the assumption that $N^\ast Y$ is Abelian. Nevertheless, when $Y$ is a point, the methods for proving Theorem \ref{16:02:38} are readily adjusted to re-obtain Theorem \ref{theo:Bailey}. Interestingly though, this adjustment only works when $Y$ is a point.

\begin{thm}[\cite{MR3128977}]\label{16:05:53}
Let $(M,\J)$ be a generalized complex manifold and $y\in M$ a point in the complex locus of $\J$. Then there exists a neighbourhood $V$ of $y$ in $M$ and an exact two-form $B\in \Omega^2(V)$ such that $\J$ is $B$-field equivalent  to a holomorphic Poisson structure on $V$. 
\end{thm}

It is worth mentioning that the proof of Theorem \ref{16:05:53} will produce only a two-form. In contrast, the proof in \cite{MR3128977} produces a two-form and a diffeomorphism, and the resulting complex structure of the holomorphic poisson structure is already linearized. Of course, we can obtain that diffeomorphism by an application of the Newlander-Nirenberg theorem \cite{MR0088770}. 

\subsection{Codimension-one Abelian Poisson branes}

The good thing about Theorem \ref{16:02:38} is that the conditions are all phrased in terms of the holomorphic vector bundle $NY$ and do not refer to the ambient manifold $(M,\J)$. In practise though it may be difficult verifying whether there exists a 2- and 3-convex neighbourhood $U\subset NY$ and whether $H^{0,2}(U)=0$. We present here an important special case where the situation can be considerably simplified.

Suppose that $Y$ has (complex) codimension 1 and that the line bundle $NY$ is negative in the sense of Example \ref{11:18:37}. It follows from that example that $NY$ contains a neighbourhood $U$ which is $q$-convex for all $q\geq 1$. We interpret $Y$ as a complex submanifold of $U$ of codimension 1, denote by $\mcO_U$ the sheaf of holomorphic functions on $U$ and by $\mf{m}\subset \mcO_U$ the ideal of functions vanishing on $Y$. For every integer $k\geq 0$ we have an exact sequence
$$0\rightarrow \mf{m}^{k+1} \rightarrow \mf{m}^k \rightarrow \mf{m}^k/\mf{m}^{k+1}\rightarrow 0, $$  
where $\mf{m}^0=\mcO_U$. Since $\mf{m}^k/\mf{m}^{k+1}\cong i_\ast (N^\ast Y)^{\otimes k}$, where $i:Y\rightarrow U$ denotes the inclusion, the associated long exact sequence yields
$$\ldots \rightarrow H^1(Y,(N^\ast Y)^{\otimes k}) \rightarrow H^2(U,\mf{m}^{k+1}) \rightarrow H^2(U,\mf{m}^k) \rightarrow H^2(Y,N^\ast Y)^{\otimes k})\rightarrow \ldots $$ 
In particular, when $H^1(Y,N^\ast Y^{\otimes k})=H^2(Y,N^\ast Y^{\otimes k})=0$ for all $k\geq 0$ we see that $H^2(U,\mcO_U)\cong H^2(U,\mf{m}^k)$ for all $k\geq 0$. By the vanishing theorem \cite[Satz 2 (page 357)]{MR0137127}, $H^2(U,\mf{m}^k)=0$ for $k$ sufficiently large, yielding $H^{2}(\mcO_U)=0$. Combining this discussion with Theorem \ref{16:02:38}, we obtain the following result.  

\begin{thm}\label{15:25:36}
Let $(Y,\tau)\subset (M,\J)$ be an Abelian Poisson brane of complex codimension 1. Suppose that $NY$ is a negative line bundle and that $H^1(Y,N^\ast Y^{\otimes k})=H^2(Y,N^\ast Y^{\otimes k})=0$ for every integer $k\geq 0$. Then there is a neighbourhood of $Y$ in $M$ on which $\J$ is $B$-field equivalent to a holomorphic Poisson structure. 
\end{thm}  

\subsection{Interpreting $H^{0,2}$ as the space of obstructions}

We give some examples of generalized complex structures which fail to admit holomorphic gauges, due to failure of the hypotheses of either Theorem \ref{14:11:52} or Theorem \ref{16:02:38}. The Abelianity condition should be seen, less as an obstruction, and more as a way of choosing a canonical holomorphic structure on the normal bundle, and the 2-- and 3-convexity conditions are just so that we can do Hodge theory and therefore deformation theory. The condition on $H^{0,2}$, however, is more along the lines of the vanishing of an obstruction map, and readily furnishes counterexamples.

\begin{ex}[Twisted complex structure on a compact manifold]\label{gerbe example}
Let $S$ be a compact complex manifold for which the map $H^2(S) \to H^{0,2}(S)$ (taking the $(0,2)$--part of any representative) is \emph{not} surjective. Notice, in particular, that such $S$ will not be K\"ahler.

Let $\epsilon_3 \in \Omega^{0,2}(S)$ represent a class in $H^{0,2}(S)$ not in the image of $H^2(S)$. Then the given complex structure on $S$ deforms as a generalized complex structure by
\begin{align}
\epsilon = 0 + 0 + \epsilon_3.
\end{align}

Because $H^2(S) \to H^{0,2}(S)$ is not surjective, Theorem \ref{14:11:52} does not apply. Indeed, in this case, the deformed generalized complex structure does not have a holomorphic gauge. A holomorphic gauge corresponds to a deformation whose $(0,2)$ part vanishes, and a $B$-field acting on the deformation $\epsilon$ just acts by adding its $(0,2)$ part---since we assumed $[\epsilon_3]$ was not in the image of $H^2(S)$, there is no such real closed $B$-field.

For a concrete example of such a complex manifold, $S$, we can take $SU(5)$. The basic topology of the classical Lie groups is well understood. In the case at hand, $SU(5)$ is known to be 2-connected and then  Hurewicz's theorem tells us that $H_2(SU(5)) \iso \pi_2(SU(5)) \iso 0$. But there is a complex structure on $SU(5)$ for which $H^{0,2}(SU(5)) \iso \C$ (see \cite{ivanov1999}, an application of the techniques in \cite{ALEXANDROV2001251}).

\end{ex}

\begin{ex}
The above example can be extended to the case of a neighbourhood of a brane. Let $S$ and $\epsilon_3$ be as above, and let $X = S \times \C$, with $w$ a coordinate for the second factor. Then we deform the complex structure on $X$ via the  $(0,2)$--form $w \epsilon_3$. Note that, along $S \times \{0\}$, where $w=0$, the generalized complex structure is unchanged by the deformation, so the standard splitting of $TS \dsum T^*S$ gives a brane structure on $S \subset X$.

We remark that the Poisson structure, being trivial, is Abelian about $S$, and $X$ is also 2-- and 3--convex. However, Theorem \ref{16:02:38} fails to apply because $H^{0,2}(X) \neq 0$. Indeed, we can see that there is no $B$-transform which puts the deformed structure into a holomorphic gauge:

Suppose, to the contrary, that $B$ is a real, closed 2-form on $X$ with 
$B^{0,2} = -w\epsilon_3$ (thus, $B$-transforming the deformed structure by this 2-form would cancel $w\epsilon_3$ and hence yield a holomorphic gauge). Because $B$ is closed, $0=(dB)^{1,2} = \bar\del B^{1,1} - \del (w\epsilon_3)$. We can differentiate this equation in the direction of the vector field $\del_w+\del_{\bar w}$, and using that the complex structure on $X=S\times \C$ is constant in the $\C$ direction, we obtain
\begin{align}\label{dB1,2}
0 &= \Lie_{\del_w+\del_{\bar w}} \bar\del B^{1,1} - \Lie_{\del_w+\del_{\bar w}} \del(w\epsilon_3) \notag\\
&= \bar\del \Lie_{\del_w+\del_{\bar w}} B^{1,1} - \del\epsilon_3.
\end{align}
Denote by $i : S\cong S\times \{0\} \rightarrow X$ the inclusion, and define a real two-form on $S$ by
\begin{align*}
\tilde{B} :=-\bar\ve_3 +i^*\left( \Lie_{\del_w+\del_{\bar w}}B^{1,1} \right) - \epsilon_3.
\end{align*}
Then $\tilde{B}^{0,2} = -\epsilon_3$, while from \eqref{dB1,2} we see that $d\tilde{B} = 0$. But, by our assumption that $[\epsilon_3]$ was not in the image of $H^2(S)$, such a $\tilde{B}$ should not exist.
\end{ex}

\section{Applications}\label{Applications}

\subsection{The complex locus in four dimensions}\label{4d section}

As a first application of our  results, we will show in Theorem \ref{4d result} that the complex locus of any four-dimensional generalized complex manifold has a holomorphic neighbourhood. To this end, we first prove some needed results about gluing different holomorphic gauges in a neighbourhood of a point. Note that these preliminary results are valid in all dimensions, it's only in Theorem \ref{4d result} itself that we require the ambient manifold to be $4$-dimensional. 

\subsubsection{Interpolation lemmas}

The basic tools for our four-dimensional result regard the interpolation between holomorphic gauges of a generalized complex structure, as developed in \cite{2014arXiv1403.6909B}. The first result is a local interpolation lemma.

\begin{lem}[{\cite[Theorem 2.1]{2014arXiv1403.6909B}}]\label{BG lemma}
Suppose $\JJ$ is a generalized complex structure which is of complex type at a point $x$, and suppose $B_0$ and $B_1$ are two holomorphic gauges for $\JJ$. Then, in a sufficiently small neighbourhood of $x$, there is a smooth family, $B_t,\, t\in [0,1]$, of holomorphic gauges interpolating between $B_0$ and $B_1$.
\end{lem}

The second result is a global version of the lemma above. To state it we need a little setup. Given a holomorphic Poisson structure $(I,\sigma)$, with $\sigma = -\frac{1}{4}(IP + iP)$, and a smooth family of real functions $f_t,\, t\in[0,1]$, there are two families of generalized complex structures one can construct, and both of them rely on the flow of the $P$-Hamiltonian vector field of $f_t$,  $X_t = P(df_t)$.  The first, which we denote by $\J_{\sigma_t}$, is obtained by flowing $(I,\sigma)$ by $X_t$, that is, $\J_{\sigma_t}$ is a holomorphic Poisson structure $(I_t,\sigma_t)$, which solves the ODE
\begin{align}\label{IODE}
\begin{cases}
(I_0,\sigma_0) &= (I,\sigma)\\
\frac{d}{dt}(I_t,\sigma_t) &= \mathcal{L}_{X_t} (I_t,\sigma_t).
\end{cases}
\end{align}
The second structure, which we denote by $\J_{B_t}$, is the $B$-field transform of $\J_\sigma$ by the real closed 2-form $B_t$ which solves the initial value problem
\begin{equation}\label{eq:B_t dot}
\begin{cases}
B_0 & =0\\
\frac{d}{dt} B_t & = -2i\del_t\bar\del_t f_t,
\end{cases}
\end{equation}
where $\del_t\bar\del_t$ is calculated with respect to the complex structure $I_t$, the solution to \eqref{IODE}.

\begin{lem}\label{flow equals gauge}
Let $(I,\sigma)$ be a holomorphic Poisson structure, with $\sigma = -\frac{1}{4}(IP + iP)$ and $\J_\sigma$ be the corresponding generalized complex structure. For any family of smooth real functions $f_t,\, t\in[0,1]$, the generalized complex structures $\J_{\sigma_t}$ and $\J_{B_t}$ agree.

Conversely, if $B_t$ is a family of real closed 2-forms, with $B_0 =0$, which are holomorphic gauges for $\J_{\sigma}$, that is, $\J_{B_t} = \J_{\tilde\sigma_t}$ for a (unique) family of holomorphic Poisson structures $(\tilde I_t,\tilde\sigma_t)$ and if $\frac{d B_t}{dt} = -2i \del_t \delbar_t f_t$ for some family of real functions, then $(\tilde I_t,\tilde\sigma_t)$ is the $P$-Hamiltonian flow by  $f_t$ of $(I,\sigma)$.
\end{lem}
\begin{proof}
We start with the first statement. Given $f_t$ and letting $(I_t,\sigma_t)$ and $B_t$ be as above, we have that $\sigma_t = -\frac{1}{4}(I_tP + iP)$ and hence $\J_{\sigma_t}$ and $\J_{B_t}$ are given by
\[
\J_{\sigma_t} = \begin{pmatrix}
-I_t & P \\
0& I_t^* 
\end{pmatrix},\qquad
\J_{B_t} = \begin{pmatrix}
-(I + PB_t) & P \\
-(B_t I + (I^* + B_tP)B_t )& I^* + B_t P 
\end{pmatrix}.
\]
The automorphism part of $\J_{B_t}$, $\tilde I_t : = I + PB_t$, therefore satisfies $\tilde I_0 = I$ and
\begin{equation}\label{eq:fundamental relation}
\tfrac{d\tilde I_t}{dt} = \frac{d}{dt}PB_t = -2iP \del_t\delbar_t f_t = -(4\sigma_t \del_t\delbar_t f_t  + c.c.) = 2i\delbar_t X_t^{1,0} + c.c. = \mathcal{L}_{X_t}I_t,
\end{equation}
where $c.c.$ stands for the complex conjugate of the term preceding it and we have used that, by construction,  $\mathcal{L}_{X_t}I_t =  2i\delbar_t X_t^{1,0} + c.c.$. Therefore $\tilde I_t$ and $I_t$ are solutions to the same initial value problem and hence must be the same endomorphism.

As a consequence we have that  $C_t$, the 2-form part of $\J_{B_t}$,  is given by
\begin{align*}
C_t &= B_tI + (I^* + B_tP) B_t = B_t(I+PB_t) + (I^* + B_tP) B_t - B_t PB_t\\& = B_tI_t+ I_t^*B_t - B_t PB_t.
\end{align*}
It is clear that $C_0 = 0$ and we have further
\begin{align*}
\tfrac{dC_t}{dt} &= \frac{d}{dt}(B_tI_t + I_t^* B_t - B_t P B_t)\\
& = \dot B_tI_t + B_t \dot I_t + \dot I_t^* B_t + I_t^* \dot B_t - \dot B_t P B_t - B_t P \dot B_t\\
& =  B_t \dot I_t + \dot I_t^* B_t - \dot B_t P B_t - B_t P \dot B_t\\
& =  B_t P\dot B_t + \dot B_t P  B_t - \dot B_t P B_t - B_t P \dot B_t\\
&=0,
\end{align*}
where in the third equality we used that $\dot B_t$ is of type (1,1) for $I_t$, hence $\dot B_tI_t + I_t^* \dot B_t =0$ and in the fourth equality we used that, as we have established,  $I_t = I + PB_t$. From this we conclude that $C_t \equiv 0$ and $\J_{B_t} = \J_{\sigma_t}$.

Conversely, given a family, $B_t$, of holomorphic gauges  for $\J_\sigma$ such that $\dot B_t = -2i \del_t\delbar_t f_t$ with respect to the corresponding complex structure, $\tilde I_t = I + P B_t$, then  \eqref{eq:fundamental relation} shows that $\tilde I_t$ satisfies the same differential equation as $I_t$, the $P$-hamiltonian flow of $I$ by $f_t$, hence they agree and
$$\J_{\sigma_t} = \J_{B_t}  = \begin{pmatrix} -I_t & P \\ 0 & I^*_t\end{pmatrix}$$
which is the $P$-Hamiltonian flow of $\J_\sigma$.
\end{proof}

%
%
%

\subsubsection{Gluing different holomorphic gauges about a point}

\begin{prop}\label{gluing prop}
Suppose that $\JJ$ is a generalized complex structure which is of complex type at point $x$, and suppose that $B_0$ and $B_1$ are holomorphic gauges for $\JJ$. Then, for any neighbourhood $V$ of $x$, there exists a holomorphic gauge $B$ and a neighbourhood $U \subset V$ of $x$ such that $B = B_0$ outside of $V$ and $B = B_1$ inside of $U$.
\end{prop}

\begin{proof}
In fact, we will construct a family, $C_t,\, t\in[0,1]$, of holomorphic gauges such that $C_0 = B_0$ and $C_1$ equals the desired $B$ of the Proposition.

First, we invoke Lemma \ref{BG lemma}, to produce a family, $B_t$, of holomorphic gauges interpolating between $B_0$ and $B_1$, defined on some neighbourhood $V' \subset V$ of $x$. We call the corresponding family of holomorphic Poisson structures $(I_t,\sigma_t)$, with $\sigma_t = -\frac{1}{4}(I_tP + iP)$. The complex Poincar\'e lemma allows us to find (on a sufficiently small neighbourhood $V'' \subset V'$ of $x$) a smooth family of $\del\bar\del$--primitves for $\dot{B}_t$, i.e., a smooth family of functions, $f_t$, such that
\begin{align}\label{idd1}
\frac{d}{dt}B_t = -2i\del_t\bar\del_t f_t
\end{align}
(where $\del_t\bar\del_t$ is always calculated with respect to the complex structure, $I_t$, of the holomorphic gauge $B_t$). By Lemma \ref{flow equals gauge}, we conclude that the family $(I_t,\sigma_t)$ is also generated by the $P$-Hamiltonian flow of $f_t$.

Next, we define  $h_t  = \psi f_t$, where $\psi$ is a smooth function which vanishes outside of $V''$ and which equals $1$ on some neighbourhood of $x$. The function $h_t$ generates a family of $P$-Hamiltonian diffeomorphisms about $x$ and, by pushing forward $(I_0,\sigma_0)$ through this family, we get a family, $(J_t,\beta_t)$ of holomorphic Poisson structures, with $(J_0,\beta_0) = (I_0,\sigma_0)$, related by this $P$-Hamiltonian flow. By Lemma \ref{flow equals gauge}, we conclude that the $(J_t,\beta_t)$ are also related by a family of $B$-transforms $e^{C_t}$, where
\begin{align}\label{idd2}
\frac{d}{dt} C_t = i\del_t\bar\del_t h_t
\end{align}
(where $\del_t\bar\del_t$ is taken with respect to $J_t$).

Since $P$ vanishes at $x$, $x$ is a fixed point of both Hamiltonian flows. Then since $h_t = f_t$ on a neighbourhood of $x$, there is some neighbourhood $U \subset V''$ of $x$ on which both Hamiltonian flows---and thus the two families $(I_t,\sigma_t)$ and $(J_t,\beta_t)$---are equal for all $t \in [0,1]$.

Therefore, on $U$, the right hand sides of \eqref{idd1} and \eqref{idd2} are equal, and thus so are $B_t$ and $C_t$. Setting $B = C_1$, we have a holomorphic gauge for which $B = B_1$ on $U$. But since $h_t = 0$ outside of $V$, we also have that $C_1 = C_0 = B_0$ outside of $V$.
\end{proof}

\begin{cor}\label{gluing intersections}
Suppose that $Y_0$ and $Y_1$ are closed subsets of generalized complex manifold $(M,\II)$ intersecting at isolated points of complex type, and suppose that $Y_0$ and $Y_1$ have neighbourhoods $W_0$ and $W_1$ with holomorphic gauges $B_0$ and $B_1$. Then there exists a neighbourhood $W$ of $Y_0 \cup Y_1$ with a common holomorphic gauge $B$ which equals $B_0$ near $Y_0$ and, except in an arbitrarily small neighbourhood of $Y_0 \cap Y_1$, $B_1$ near $Y_1$.
\end{cor}

\begin{proof}

Let $\{x_i\}_{i}$ denote the isolated set of intersection points between $Y_0$ and $Y_1$, let $\{Z_i\}_i$ be a mutually disjoint set of open neighbourhoods of the $x_i$ with $Z_i \subset W_0\cap W_1$ and let $V_i$ be slightly smaller neighbourhoods with $\overline{V_i}\subset Z_i$. By Proposition \ref{gluing prop} (applied to the generalized complex manifolds $Z_i$) there exists smaller open neighbourhoods $\{U_i\}_i$ with $U_i \subset V_i$, and holomorphic gauges $B_i$ on $Z_i$ with the property that $B_i=B_1$ on $Z_i\backslash V_i$ and $B_i=B_0$ on $U_i$. Now let $\tilde{W}_0 \subset W_0$ and $\tilde{W}_1 \subset W_1$ be smaller neighbourhoods of $Y_0$ and $Y_1$ with the property\footnote{These can be constructed as follows. Let $U:=\cup U_i$ and choose disjoint open neighbourhoods of $Y_0\cap U^c$ and $Y_1\cap U^c$ (which are disjoint closed subsets of a manifold). The union of these with $U$ itself gives the desired $\tilde{W}_0$ and $\tilde{W}_1$.} that $\tilde{W}_0 \cap \tilde{W}_1 \subset \cup_iU_i$. We can now construct the desired $2$-form $B$ on $W:=\tilde{W}_0 \cup \tilde{W}_1$ as follows. On $\tilde{W}_0$ we let it be $B_0$, on $\tilde{W}_1\backslash \cup_i\overline{V_i}$ we let it be $B_1$ and on $\tilde{W}_1\cap Z_i$ we let it be $B_i$. This is well-defined since firstly $B_i=B_1$ on $Z_i\backslash V_i$, and secondly because $\tilde{W}_0\cap\tilde{W}_1\subset\cup_iU_i$ and on each $U_i$ we know that $B_i=B_0$. 
\end{proof}

\subsubsection{A neighbourhood theorem for the complex locus in four dimensions}

\begin{lem}\label{h^top on vb}
Let $Y$ be complex manifold and $K \to Y$ a holomorphic vector bundle. Let $h$ be an Hermitian metric on $K$ and $D:=\{z \in K| |z| \leq 1\}$ the associated disc bundle. Then 
$H^{0,n+k}(D) = 0$, where $n:=\dim_\C Y$ and $k:=\rank_\C K$.
\end{lem}

\begin{proof}
Let $\beta \in \Omega^{0,n+k}(D)$ be $\delbar$--closed and let $(\mathbf{w},\mathbf{z}) := (w_1,\ldots,w_k,z_1,\ldots,z_n)$ be coordinates for a local trivialization of $K \to Y$, with the zero section $Y \subset K$ corresponding to $\mathbf{w}=0$, and $\mathbf{z}$ being coordinates along $Y$. Then
\begin{align}
\beta \,=\, f \, d\bar{\mathbf{w}} \^ d\bar{\mathbf{z}} \,:=\, f\, d\bar{w}_1 \^ \ldots \^ d\bar{w}_k \^ d\bar{z}_1 \^ \ldots \^ d\bar{z}_n
\end{align}
for some smooth function $f$. Locally in $Y$, we can find a $\delbar$-primitive for $\beta$ by finding a $\delbar$ primitive, $\alpha$, for $f\,d\bar{\mathbf{w}}$ along each fiber, and taking the $(0,n+k-1)$--form $\alpha \^ d\bar{\mathbf{z}}$.

Choose a smooth partition of unity subordinate to a nice cover of $Y$ (whose open sets have coordinate charts).  This pulls back to a partition of unity over $D$, whose bump functions we call $\phi^{(i)}$, with $i$ indexing the open set. Let $\mathbf{z}^{(i)}$ be the corresponding local coordinates on $Y$, and $\alpha^{(i)} \wedge d\bar{\mathbf{z}}^{(i)}$ the corresponding local primitive for $\beta$, constructed as above. Since $\phi^{(i)}$ comes from $Y$ we have $\delbar \phi^{(i)} \wedge d\bar{\mathbf{z}}^{(i)} = 0$ and hence the globally-defined form $\sum_i \phi^{(i)} \alpha^{(i)} \wedge d\bar{\mathbf{z}}^{(i)}$ is a $\delbar$--primitive for $\beta$.
\end{proof}

We can finally prove the promised unconditional result about generalized complex 4-manifolds:

\begin{thm}\label{4d result}
	Suppose $Y$ is the compact complex locus inside a generalized complex 4-manifold $(M,\JJ)$. Then $\JJ$ is equivalent to a holomorphic Poisson structure in a neighbourhood of $Y$ in $M$.
\end{thm}

\begin{proof}
Locally about any point, $Y$ looks like the vanishing locus of a holomorphic section of a line bundle (the anticanonical bundle) on $\C^2$. Therefore it has components of complex dimensions $1$ and $2$. On any component of complex dimension 2, the result is trivial; so we assume $Y$ locally has the structure of a complex curve.

However, $Y$ need not be smooth. Therefore, we will take a certain resolution of $Y$ and $M$, and we shall see that Theorem \ref{16:02:38} does in fact apply to this resolution. Then we will use the earlier results of this section to show that the local holomorphic gauge on the resolution passes to $Y$ in $M$.

The authors studied generalized complex blow-ups previously in \cite{MR3894047}. In 2 complex dimensions the blow-up of a point of complex type, $y \in Y \subset M$, can be described as follows: Using Theorem \ref{16:05:53}, choose a local holomorphic gauge in a neighbourhood of $y \in M$. Then blow up $M$ at $y$ just as a complex manifold, giving $\pi : \tilde{M} \to M$ with exceptional divisor $E \subset \tilde{M}$ and $\pi : \tilde{M} \setminus E \to M \setminus \{y\}$ an isomorphism. By a result of Polishchuk \cite{MR1465521}, the holomorphic Poisson structure about $y$ pulls back to a holomorphic Poisson structure about $E$. Then the generalized complex structure on $M \setminus \{y\}$ pulls back to $\tilde{M} \setminus E$, which then glues to the pullback holomorphic Poisson structure to give a generalized complex structure on $\tilde{M}$. It is true that this blow-up does not depend on the choice of gauge used to describe it \cite{MR3894047}, but  this fact is unimportant here.

Singularities of a complex curve may be resolved by, at each singular point, applying a sequence of blow-ups. Thus, we apply the above blow-up construction recursively to the singular points, $\{y_1,y_2,\ldots\}$, of $Y$ and its blow-ups, giving a full resolution,
\begin{align}
\pi : \tilde{M} \to M,
\end{align}
with exceptional divisor $E = \pi^{-1}(\{y_1,y_2,\ldots\}) \subset \tilde{M}$ and
\begin{align}\label{blowdownmap}
\pi : \tilde{M} \setminus E \to M \setminus \{y_1,y_2,\ldots\}
\end{align}
an isomorphism, such that the smooth proper transform\footnote{This is by definition the closure of $\pi^{-1}(Y)$, where $\pi$ is the map given by \eqref{blowdownmap}.} $\tilde{Y} \subset \tilde{M}$ surjects onto $Y$. Note that the generalized complex structure, $\tilde{\JJ}$, on $\tilde{M}$ comes with a choice, $B_0$, of holomorphic gauge in a neighbourhood of $E$, coming from the fact that, in this neighbourhood, $\tilde{\JJ}$ is constructed just as a holomorphic Poisson blow-up.

By construction, the proper transform $\tilde{Y} \subset \tilde{M}$ is a smooth complex curve. Any generalized complex structure of complex type in 1 complex dimension is automatically of block diagonal form, and thus in a holomorphic gauge---therefore $\tilde{Y}$ carries a brane structure. Furthermore, the unit disc bundle $D\subset N\tilde{Y}$ (with respect to an auxiliary Hermitian metric) is automatically 2- and 3-convex since $N\tilde{Y}$ has rank $1$, and satisfies $H^{0,2}(D)=0$ by Lemma \ref{h^top on vb}. Finally, since $\tilde{Y}$ has complex codimension 1, it is an Abelian brane. Thus, Theorem \ref{16:02:38} applies, and we have a holomorphic gauge, $B_1$, in some neighbourhood of $\tilde{Y}$ in $\tilde{M}$.

The gauges $B_1$ and $B_0$ may not agree where they are both defined. However, by Corollary \ref{gluing intersections}, there is another holomorphic gauge, $B$, defined in a neighbourhood of $\tilde{Y} \cup E$, which equals $B_0$ in a neighbourhood of $E$. Since $\pi : \tilde{M} \setminus E \to M \setminus \{y_1,y_2,\ldots\}$ is an isomorphism, away from $E$, $B$ is a pullback of a holomorphic gauge in a neighbourhood of $Y \setminus \{y_1,y_2,\ldots\}$. But near $E$, $B = B_0$ is, by construction, the pullback of a holomorphic gauge in a neighbourhood of $\{y_1,y_2,\ldots\}$. Therefore, $B$ is a pullback of a holomorphic gauge in a neighbourhood of $Y$.
\end{proof}

\subsection{Blowing down submanifolds}

In \cite{MR3894047} the authors extended the technique of blowing up submanifolds from symplectic and complex geometry to generalized complex geometry. One of the results in \cite{MR3894047} states that if $Y\subset (M,\J)$ is a generalized Poisson submanifold such that the Lie algebra structure on $N^\ast Y$ is degenerate\footnote{A lie algebra $\mfg$ is \emph{degenerate} if the map $\Lambda^3\mfg \rightarrow S^2\mfg$ given by $x\wedge y\wedge z\mapsto x[y,z]+y[z,x]+z[x,y]$ vanishes.}, then $Y$ can be blown up in a generalized complex manner. What this entails exactly is explained in \cite{MR3894047}, but roughly it means that there is a generalized complex manifold $(\widetilde{M},\widetilde{\J})$ together with a \emph{blow-down} map $p:\widetilde{M}\rightarrow M$, inducing a generalized complex isomorphism 
$$\widetilde{M}\backslash p^{-1}(Y) \cong M\backslash Y$$
with the property that the \emph{exceptional divisor} $\widetilde{Y}:=p^{-1}(Y)$ is isomorphic to $\mathbb{P}(NY)$ as a smooth fiber bundle over $Y$. Here $\mathbb{P}(NY)$ denotes the complex projectivization of the complex vector bundle $NY$. In particular, blowing up $Y$ in $M$ effectively replaces $Y$ with a codimension one submanifold $\widetilde{Y}$, within the category of generalized complex manifolds. Conversely we can ask whether a given codimension 1 submanifold $\widetilde{Y}\subset (\widetilde{M},\widetilde{\J})$ may be replaced by a submanifold of larger codimension. This notion is referred to as \emph{blowing-down}, and plays an important role for instance in the minimal model program in algebraic geometry (where the process is also referred to as \emph{contraction}). The rigorous definition of blowing down a submanifold as well as finding concrete examples is more difficult than for blowing up, but using our normal form results obtained thus far we can prove the following.

\begin{thm}\label{theo:blow-down}
Let $(\widetilde{M},\widetilde{\J})$ be a generalized complex manifold and $\widetilde{Y}\subset (\widetilde{M},\widetilde{\J})$ an Abelian Poisson brane. Suppose that $\widetilde{Y}$ is diffeomorphic to $\C\mathbb{P}^{n-1}$, that $N\widetilde{Y}\cong \mathcal{O}_{\C\mathbb{P}^{n-1}}(-1)$ as complex line bundles, and that the induced complex structure on $\widetilde{Y}$ coincides with the standard complex structure on $\C\mathbb{P}^{n-1}$. Then there exists a generalized complex manifold $(M,\J)$ and a point $y\in M$ such that $(\widetilde{M},\widetilde{\J})$ is the blow-up of $Y=\{y\}$ in $M$ with exceptional divisor $\widetilde{Y}$.   
\end{thm}
\begin{proof}
By the classification of holomorphic line bundles on $\C\mathbb{P}^{n-1}$ we know that $N\widetilde{Y}\cong \mathcal{O}_{\C\mathbb{P}^{n-1}}(-1)$ as holomorphic line bundles and hence that $H^1(\widetilde{Y},N\widetilde{Y}^{\otimes k})=H^2(\widetilde{Y},N\widetilde{Y}^{\otimes k})=0$ for all $k\geq 0$. By Theorem \ref{15:25:36} we can apply a global $B$-field transform to $(\widetilde{M},\widetilde{\J})$ after which a neighbourhood of $\widetilde{Y}$ is given by a holomorphic Poisson structure. According to \cite[Satz 7 (page 363)]{MR0137127} a neighbourhood of $Y$ is isomorphic, as a complex manifold, to a neighbourhood of the zero section in $\mathcal{O}_{\C\mathbb{P}^{n-1}}(-1)$. This neighbourhood can therefore be blown down, as a complex manifold, reducing $Y$ to a point. Finally we observe that \cite[Proposition 8.4]{MR1465521} implies that the Poisson bivector descends to a unique Poisson structure on the blow-down. Since the generalized complex structure is given by the holomorphic Poisson structure around $Y$, the blow-down is one of generalized complex manifolds as well.
\end{proof}
\begin{rem}
Basically we used the neighbourhood theorem Theorem \ref{16:02:38} to reduce the blowing down problem to the category of complex manifolds, and this strategy also applies if we want to blow down to something different than a point. Note that the condition on the complex structure on $\C\mathbb{P}^{n-1}$ is non-trivial for $n\geq 3$ as it is unknown in general whether the complex structure on projective space is unique.
\end{rem}
\begin{cor}
Let $Y$ be a real two-dimensional surface in the complex locus of a four-dimensional generalized complex manifold, diffeomorphic to $\C\mathbb{P}^1$ and with self-intersection $-1$. Then $Y$ can be blown down to a point.
\end{cor}

\section{Proof of the main results}\label{Proof}

\subsection{Leibniz bounds}\label{sec:leibnizbounds}

In order to work efficiently with the estimates in this paper we will introduce some notation, taken from \cite{MR3128977}. Let $E$ and $F$ be vector bundles over $M$ with sections $u\in C^\infty(E)$ and $v\in C^\infty(F)$. For integers $k,l\geq0$ we write
\begin{align}\label{16:18:42}
\mcL(|u|_k;||v||_l):=Poly(|u|_{\floor{k/2}+1 })\cdot ||v||_l+Poly(|u|_{\floor{k/2}+1 })\cdot |u|_{k}\cdot ||v||_{\floor{l/2}+1}. 
\end{align}
Here $Poly$ denotes a polynomial with non-negative coefficients that depends only on $k$, $l$ and the $C^0$-norm of $u$. Furthermore, $\floor{x}$ denotes the biggest integer bounded by $x$ and $|\cdot|$ and $||\cdot||$ denote the $C^k$-norms and Sobolev norms, respectively, defined with respect to some auxiliary metrics on $M$, $E$ and $F$. We think of $\mcL(|u|_k;||v||_l)$ as a polynomial that is linear in $v$ and for which every monomial contains at most one high derivative-norm (of order $k$ or $l$), the rest being low derivative norms (of order $k/2$ or $l/2$). The letter $\mcL$ stands for Leibniz, on account of the estimate
$$||uv||_k=\sum_{|\alpha|\leq k} ||\del^\alpha (uv)||\leq \sum_{|\alpha|\leq k}\sum_{\beta+\gamma=\alpha} |\del^\beta u|\cdot || \del^\gamma v||\leq \mcL(|u|_k;||v||_k)$$
for functions $u,v\in C^\infty(M)$. Here are some more examples to illustrate how we will use this notation.
\begin{ex}\label{16:37:47}
Let $\ve_2\in \Omega^{0,1}(T^{1,0}M)$ denote a deformation of an almost complex structure $I$ and let $\varphi\in \Omega^{0,q}(M)$. Then we have
$$||\delbar_{\ve_2}\varphi||_k=||\delbar\varphi+\lb\ve_2,\varphi\rb||_k\leq \mcL(|\ve_2|_{k+1};||\varphi||_{k+1}). $$
Similarly, if $\Delta_{\ve_2}=\delbar_{\ve_2}^\ast\delbar_{\ve_2}+\delbar_{\ve_2}\delbar_{\ve_2}^\ast$ denotes the Laplacian, then 
\begin{align}\label{16:58:22}
||\Delta_{\ve_2}\varphi||_k\leq \mcL(|\ve_2|_{k+2};||\varphi||_{k+2}). 
\end{align}
 Note that in this example the left-hand side is linear in $\varphi$ but not in $\ve_2$. 
\end{ex}
\begin{rem}
When dealing with estimates like \eqref{16:58:22} we will often replace $|\ve_2|_{k+2}$ and $||\varphi||_{k+2}$ by $|\ve_2|_{k+b}$ and $||\varphi||_{k+b}$ where $b$ is some fixed large integer. The important thing is that $b$ is fixed and independent of $k$.
\end{rem} 
We will also make use of $\mcL(|u|_k;|v|_k)$, which is given by the same expression as \eqref{16:18:42} but with $||v||$ replaced by $|v|$. 
\begin{ex} Combined with the Sobolev inequality \eqref{16:58:22} yields the estimate
\begin{align*}
|\Delta_{\ve_2}\varphi|_k\leq \mcL(|\ve_2|_{k+b};|\varphi|_{k+b}),
\end{align*}
for some integer $b$ (which incorporates the dimension of $M$).
\end{ex}

\subsection{Proof of the rigidity theorem}\label{15:39:45}

This subsection is devoted to the proof of Theorem \ref{14:11:52}, which is stated again below but more quantitatively. In what follows we will fix auxiliary Hermitian metrics on our manifolds and vector bundles. We will also use the terminology of Section \ref{14:20:32}, in particular we will use the action of two forms on deformations and the decomposition (\ref{14:30:49}) for generalized complex deformations of a complex structure. 

\begin{thm}\label{14:17:27}
Let $(M,I,\sigma)$ be a compact holomorphic Poisson manifold with boundary such that $(M,I)$ is $2$- and $3$-convex and for which $H^{0,2}(M,I)=0$. Then there exists an integer $a\in \mathbb{Z}_{>0}$ and a real constant $\delta>0$ with the following property. For any generalized complex deformation $\ve$
of $(I,\sigma)$ that satisfies 
\begin{align}\label{16:55:42}
|\ve|_a<\delta,
\end{align}
there exists a one-form $\xi\in \Omega^1(M)$ with the property that $(e^{d\xi}\cdot \ve)_3=0$.
\end{thm}

The proof of Theorem \ref{14:17:27} starts on page \pageref{16:29:09}. We will first sketch the overall strategy, provide some necessary lemmas and set up a Nash-Moser type algorithm that will be used in the proof.

\subsubsection*{Strategy} 

We will construct the one-form of Theorem \ref{14:17:27} as an infinite sum $\xi=\sum_{i=0}^\infty \xi^{(i)}$ using an iterative algorithm. The idea is to construct $\xi^{(0)}$ from $\ve^{(0)}:=\ve$ so that $\ve^{(1)}:=e^{d\xi^{(0)}}\cdot \ve^{(0)}$ satisfies\footnote{We will give precise meanings to these kind of inequalities later by using $C^k$-norms.} $\ve^{(1)}_3<\ve^{(0)}_3$. Then we construct $\xi^{(1)}$ from $\ve^{(1)}$ so that $\ve^{(2)}:=e^{d\xi^{(1)}}\cdot \ve^{(1)}$ satisfies $\ve^{(2)}_3<\ve^{(1)}_3$, and so on. Setting things up correctly, the sum $\xi=\sum_{i} \xi^{(i)}$ will converge in $C^\infty(T^\ast M)$ and satisfy
\begin{align*}
(e^{d\xi}\cdot \ve)_3=\lim_{n\rightarrow \infty} (e^{d\sum_{i=0}^n\xi^{(i)}}\cdot \ve)_3=\lim_{n\rightarrow \infty} \ve^{(n+1)}_3=0.
\end{align*} 
The construction of $\xi^{(i)}$ from $\ve^{(i)}$ is a bit involved, mainly because $\ve^{(i)}_2$ is not necessarily integrable as an almost complex deformation of $I$. Modulo some details though, $\xi^{(i)}$ will be given by a $\delbar_{\ve^{(i)}_2}$-primitive of $\ve^{(i)}_3$ (recall that $\delbar_{\ve^{(i)}_2}\ve^{(i)}_3=0$ because $\ve$ is integrable, c.f.\ \ref{1:07:34}). The fact that we can construct such primitives follows from the Hodge theory of Theorem \ref{15:35:57} combined with the assumptions that $(M,I)$ is $2$-convex and that $H^{0,2}(M,I)=0$. The reason that we also need $3$-convexity is because of a technicality in Lemma \ref{14:51:45}. 

The precise definition of $\xi^{(i)}$ will be given later, but there is one technical aspect in its definition that deserves to be mentioned here. To ensure that $\xi=\sum_{i} \xi^{(i)}$ converges we need suitable estimates on all the $C^k$-norms $|\xi^{(i)}|_k$. Since $\xi^{(i)}$ is constructed from $\ve^{(i)}_3$ using the Neumann operator associated to ${\ve^{(i)}_2}$, (\ref{15:53:51}) bounds the $C^k$-norms of $\xi^{(i)}$ by those of $\ve^{(i)}_2$ and $\ve^{(i)}_3$ and therefore by induction by those of $\ve_2^{(0)}$ and $\ve_3^{(0)}$. There is however a shift of the norm-degrees in these estimates which becomes bigger as $i$ increases,  preventing us from expressing $|\xi^{(i)}|_k$ directly in terms of the starting data $|\ve_2^{(0)}|_k$ and $|\ve_3^{(0)}|_k$. This phenomenon is sometimes called a \lq\lq loss of derivatives\rq\rq, because information about a certain $C^k$-norm of $\ve^{(0)}$ only provides information for a smaller $C^k$-norm of $\xi^{(i)}$. To overcome this issue we have to modify the elements $\xi^{(i)}$ with the help of so-called \textsl{smoothing operators}, which compensate the loss of derivatives at each iterative step. This technique was first introduced by Nash in \cite{MR75639}.

\subsubsection*{The necessary estimates}

Throughout this section we let $(M,I,\sigma)$ denote a fixed compact holomorphic Poisson manifold with boundary. Let 
\begin{align*}
\mcD&:=C^\infty(\Lambda^2 T^{1,0}M)\oplus C^\infty(T^{\ast 0,1}M\otimes T^{1,0}M )\oplus    C^\infty(\Lambda^2 T^{\ast 0,1}M),\\
\mcG &:=C^\infty(T^\ast M ),
\end{align*} 
be the spaces parametrizing generalized complex deformations of $(I,\sigma)$ and one-forms (thought of as exact $B$-field transformations) on $M$, respectively. Denote by $|\cdot |_k$ the $C^k$-norms on these Fr\'echet spaces. Note that for the Nash-Moser algorithm we only need these $C^k$-norms and not the Sobolev norms, although the latter will appear in intermediate estimates involving Hodge theory. As explained in Section \ref{14:20:32}, the group $\mcG$ acts on the set of deformations $\mcD$  
\begin{align*} 
(\xi,\ve)\mapsto e^{d\xi}\cdot\ve,
\end{align*}
which is well defined when $|\xi|_1$ is small compared to $|\sigma+\ve_1|_0$ (c.f.\ Remark \ref{11:43:33}). 

In order to construct the sequence of one-forms of the algorithm (c.f.\ the strategy above), we will use the map $\Phi: \mcD \rightarrow \mcG$ defined by
\begin{align}\label{14:14:54} 
\Phi(\ve):=-\delbar^\ast_{\ve_2}N_{\ve_2}\ve_3 -(1+\ve_2)(1-\overline{\ve_2}\ve_2)^{-1}(\overline{\delbar^\ast_{\ve_2}N_{\ve_2}\ve_3}+\overline{\ve_2}(\delbar^\ast_{\ve_2}N_{\ve_2}\ve_3)).
\end{align}
If $(M,I)$ is $2$-convex, Theorem \ref{15:35:57} implies that there exists an integer $a_0>0$ and a constant $\delta_0>0$ such that $\Phi(\ve)$ is well defined whenever $|\ve_2|_{a_0}<\delta_0$. 

\begin{rem}
The definition of $\Phi$ is chosen to ensure that $\Phi(\ve)$ is real and satisfies
\begin{align}\label{21:28:07}
\Phi(\ve)^{0,1}-\ve_2(\Phi(\ve)^{1,0})=-\delbar^\ast_{\ve_2}N_{\ve_2}\ve_3.
\end{align} 
The left-hand side equals the $(0,1)$-part of $\Phi(\ve)$ with respect to the deformed almost complex structure $I_{\ve_2}$, while the right-hand side will turn out to give an approximate $\delbar_{\ve_2}$-primitive of $\ve_3$ whenever $\delbar_{\ve_2}\ve_3=0$ and $H^{0,2}(M,I)=0$. 
\end{rem}
\begin{lem} Suppose that $(M,I)$ is $2$-convex and that $H^{0,2}(M,I)=0$. 
Then there exist $a_0\in \mathbb{Z}_{>0}$ and $\delta_0>0$ with the following property. For all $\ve\in \mcD$ satisfying $|\ve|_{a_0} <\delta_0$ 
we have (abbreviating $\Phi:=\Phi(\ve)$)
\begin{align}\label{14:48:01}
(e^{d\Phi}\cdot\ve)_3=&\delbar^\ast_{\ve_2}\delbar_{\ve_2}N_{\ve_2}\ve_3-\lb \ve_3,(\sigma+\ve_1)(\Phi^{1,0})\rb \nonumber \\&+ \big(\ve_2d\Phi^{2,0}-d\Phi^{1,1}\big)(\sigma+\ve_1)\big(1+d\Phi^{2,0}(\sigma+\ve_1)\big)^{-1}\big(d\Phi^{1,1}+d\Phi^{2,0}\ve_2\big).
\end{align}
\end{lem}
\begin{proof} Because of 2-convexity we know that $\Phi(\ve)\in \mcG$ is well-defined for $|\ve_2|_{a_0}<\delta_0$ when $a_0$ is sufficiently large and $\delta_0$ sufficiently small. We may assume that $\delta_0<1$ so that $|\sigma+\ve_1|_0\leq |\sigma|_0+\delta_0<|\sigma|_0+1$, and because $\Phi:\mcD\rightarrow \mcG$ is continuous we may enlarge $a_0$ and shrink $\delta_0$ (depending on $|\sigma|_0$) so that $e^{d\Phi}\cdot\ve\in \mcD$ is well-defined whenever $|\ve|_{a_0}<\delta_0$ (c.f.\ Remark \ref{11:43:33}). 

Since $H^{0,2}(M,I)=0$, Theorem \ref{15:35:57} 4) implies that $\mathcal{H}_{\ve_2}^{0,2}=\text{Ker}(\Delta_{\ve_2})=0$ for $|\ve_2|_{a_0}<\delta_0$ (after possibly enlarging $a_0$ and shrinking $\delta_0$). Consequently, using (\ref{21:28:07}) we deduce that
\begin{align*}
\delbar_{\ve_2}(\Phi^{0,1}-\ve_2(\Phi^{1,0}))=-\delbar_{\ve_2}\delbar^\ast_{\ve_2}N_{\ve_2}\ve_3=-\Delta_{\ve_2}N_{\ve_2}\ve_3+\delbar^\ast_{\ve_2}\delbar_{\ve_2}N_{\ve_2}\ve_3=-\ve_3+\delbar^\ast_{\ve_2}\delbar_{\ve_2}N_{\ve_2}\ve_3.
\end{align*}
This equation, combined with \eqref{17:36:15} and \eqref{12:30:54}, implies \eqref{14:48:01}.
\end{proof}
We will think of $\Phi(\ve)$ as a one-form designed to make sure that $(e^{d\Phi(\ve)}\cdot\ve)_3$ is smaller than $\ve_3$. Eventually we will achieve this by using that the right hand side of (\ref{14:48:01}) is quadratic in $\ve_3$. Since $\Phi=\Phi(\ve)$ depends linearly on $\ve_3$ we see that this is indeed true for all terms in (\ref{14:48:01}), except possibly for the first term. We will now set out to prove that this first term can be bounded by something that is quadratic in $\ve_3$ (Lemma \ref{14:51:45} below). We will make use of the following estimates, whose proofs can be found in \cite{MR656198}.

\begin{prop}[\cite{MR656198}]\label{13:43:23}
Let $M$ be a compact manifold with boundary. 
\newline
1) For every triple of integers $0\leq k\leq l \leq m$ we have
\begin{align}
|f|_l\leq & C |f|^{\frac{m-l}{m-k}}_k |f|^{\frac{l-k}{m-k}}_m  &&\forall f\in C^\infty(M). \label{20:45:37}
\end{align}
2) For every integer $k\geq 0$ we have
\begin{align}
|fg|_k\leq & C (|f|_k|g|_0+|f|_0|g|_k) && \forall f,g\in C^\infty(M). \label{17:19:25}
\end{align} 
3) If $(i,j)$ lies on the line segment joining $(k,l)$ and $(m,n)$, we have 
\begin{align}
|f|_i|g|_j\leq & C (|f|_k|g|_l+|f|_m|g|_n) && \forall f,g\in C^\infty(M). \label{17:29:45}
\end{align} 
(In all three cases $C$ denotes a constant that depends only on $i,j,k,l,m$) 
\end{prop} 
\begin{rem}
These inequalities also hold for multiplicative operations involving tensors. For example, for a Dirac structure $L\subset \T M_\C$ we have 
\begin{align}\label{13:58:02}
| \lb u,v \rb |_k\leq C(|u|_{k+1}|v|_1+|u|_{1}|v|_{k+1}) \hspace{10mm} \forall u,v\in C^{\infty}(\Lambda^\bullet L).
\end{align} 
\end{rem}

\begin{lem}\label{14:51:45} Suppose that $(M,I)$ is 2- and 3- convex. 
Then there exist $a_0,b_0\in \mathbb{Z}_{> 0}$ and $\delta_0>0$ with the following property. 
For every $k\in\mathbb{Z}_{\geq0}$ and every integrable $\ve\in\mcD$ satisfying $|\ve_2|_{a_0}<\delta_0$, we have
\begin{align}\label{15:06:57}
|\delbar_{\ve_2}^\ast\delbar_{\ve_2}N_{\ve_2}\ve_3|_k\leq \big(\mcL(|\ve_1|_{k+b_0};|\ve_3|_{k+b_0})+\mcL(|\ve_2|_{k+b_0};|\ve_3|_{k+b_0})\big)\cdot |\ve_3|_{a_0}
\end{align}
\end{lem}
\begin{proof} For suitable $a_0$ and $\delta_0$ we know that the Neumann operator $N_{\ve_2}$ is well-defined in degrees 2 and 3 when $|\ve_2|_{a_0}<\delta_0$. Since for every $\varphi\in \Omega^{0,\bullet}(M)$ we have (see Example \ref{16:37:47})
\begin{align*}
| \delbar^\ast_{\ve_2}\varphi|_k\leq & \mcL(|\ve_2|_{k+1};|\varphi|_{k+1}),
\end{align*}
it suffices to focus our attention on $\varphi:=\delbar_{\ve_2}N_{\ve_2}\ve_3\in\Omega^{0,3}(M)$. 
We would like to apply (\ref{14:00:04}) to $\varphi$, but in order to do so we need $\varphi\in \text{Dom}(\square_{\ve_2})$. The first Neumann boundary condition  
$\sigma(\delbar^\ast_{\ve_2},dr)\varphi|_{\del M}=0$ holds because $N_{\ve_2}\ve_3\in \text{Dom}(\square_{\ve_2})$, but the second Neumann boundary condition requires that
$$\sigma(\delbar^\ast_{\ve_2},dr)\delbar_{\ve_2}\varphi|_{\del M}=\sigma(\delbar^\ast_{\ve_2},dr)\delbar^2_{\ve_2}N_{\ve_2}\ve_3|_{\del M}$$
is equal to zero, and this is not guaranteed because $\delbar_{\ve_2}^2$ is not necessarily zero. To get around this, let $r$ be a boundary-defining function as in Proposition \ref{09:59:56} 
and consider the function
\begin{align*}
f:=\sigma(\delbar^\ast_{\ve_2} ,dr)\delbar_{\ve_2}r.
\end{align*}
Since $f=-\frac{1}2$ on $\del M$ when $\ve_2=0$, $f$ is non-vanishing around $\del M$ for small $|\ve_2|_1$. Consequently, $f^{-1}$ is well-defined near $\del M$ and can be extended smoothly to all of $M$ by multiplying it with a fixed bump function that is supported near $\del M$. 
Define 
\begin{align}\label{16:55:09}
\alpha:=f^{-1}\sigma(\delbar^\ast_{\ve_2} ,dr)\delbar_{\ve_2} \varphi
=-f^{-1}\sigma(\delbar^\ast_{\ve_2} ,dr) \lb \lb \sigma+\ve_1,\ve_3\rb, N_{\ve_2}\ve_3\rb, 
\end{align}
where in the second equality we used (\ref{12:37:22}). We claim that 
$$\widetilde{\varphi}:=\varphi-r\alpha\in \text{Dom}(\square_{\ve_2}).$$ 
Indeed, the first Neumann condition holds because $\widetilde{\varphi}=\varphi$ on $\del M$. For the second Neumann condition, we use that $\sigma(\delbar^\ast_{\ve_2} ,dr)$ is a derivation and that $\sigma(\delbar^\ast_{\ve_2} ,dr)\alpha=0$ to obtain 
\begin{align*}
\sigma(\delbar^\ast_{\ve_2} ,dr)\delbar_{\ve_2}\widetilde{\varphi}|_{\del M}=\sigma(\delbar^\ast_{\ve_2} ,dr) (\delbar_{\ve_2}\varphi-\delbar_{\ve_2}r\wedge \alpha) |_{\del M}=(f\alpha-f\alpha) |_{\del M}=0.
\end{align*}
Since $\widetilde{\varphi}\in \text{Dom}(\square_{\ve_2})$ we can apply (\ref{14:00:04}) to $\widetilde{\varphi}$ which, in combination with (\ref{15:53:51}), yields
\begin{align}
|\widetilde\varphi|_k\leq & \mcL(|\ve_2|_{k+b};|\Delta_{\ve_2}\widetilde\varphi|_{k+b})+|\pi_{\ve_2}\widetilde\varphi|_k \nonumber \\
\leq & \mcL(|\ve_2|_{k+b};|\Delta_{\ve_2}\varphi|_{k+b})+ \mcL(|\ve_2|_{k+b};|\Delta_{\ve_2}(r\alpha)|_{k+b})+|\pi_{\ve_2}(r\alpha)|_k.\label{14:37:25}
\end{align}
Note that $\pi_{\ve_2}\varphi=0$ because $\text{Im}(\delbar_{\ve_2})$ is orthogonal to $\mcH_{\ve_2}$. We will estimate the above three terms individually. Since $\delbar_{\ve_2}\Delta_{\ve_2}N_{\ve_2}\ve_3=\delbar_{\ve_2}\ve_3=0$, we have
 \begin{align}
\Delta_{\ve_2}\varphi=[\Delta_{\ve_2},\delbar_{\ve_2}]N_{\ve_2}\ve_3=&
(\delbar_{\ve_2}^\ast \delbar_{\ve_2}^2-\delbar_{\ve_2}^2\delbar_{\ve_2}^\ast) N_{\ve_2}\ve_3\nonumber \\=&-\delbar_{\ve_2}^\ast\lb \lb\sigma+\ve_1,\ve_3\rb, N_{\ve_2}\ve_3\rb+\lb \lb\sigma+\ve_1,\ve_3\rb, \delbar^\ast_{\ve_2}N_{\ve_2}\ve_3\rb,\label{13:47:40}
\end{align}
where in the last step we used (\ref{12:37:22}). Next, since $\Delta_{\ve_2}$ is a second-order operator we obtain 
\begin{align}\label{13:47:23}
\mcL(|\ve_2|_{k+b};|\Delta_{\ve_2}(r\alpha)|_{k+b})\leq \mcL(|\ve_2|_{k+b+2};|r\alpha|_{k+b+2}).
\end{align}
Finally, using the Sobolev estimate and (\ref{21:18:00}) we conclude that
 \begin{align}\label{13:47:54}
|\pi_{\ve_2}(r\alpha)|_k\leq C||\pi_{\ve_2}(r\alpha)||_{k+b}\leq \mcL(|\ve_2|_{k+2b};||\pi_{\ve_2}(r\alpha)||)\leq \mcL(|\ve_2|_{k+2b};|r\alpha|). 
\end{align}
Combining (\ref{16:55:09})-(\ref{13:47:54}), we obtain the estimate
\begin{align}
|\varphi|_k\leq \mcL(|\ve_2|_{k+2b};|\Theta(\ve)|_{k+2b})
\end{align}
where
\begin{align*}
\Theta(\ve):=\lb \lb \sigma+\ve_1 ,\ve_3\rb, N_{\ve_2}\ve_3\rb+\lb \lb \sigma+\ve_1 ,\ve_3\rb, \delbar_{\ve_2}^\ast N_{\ve_2}\ve_3\rb.
\end{align*}
Applying (\ref{13:58:02}) repeatedly together with (\ref{15:53:51}) we obtain  
\begin{align*}
|\Theta(\ve)|_{k+2b}\leq \big( \mcL(|\ve_1|_{k+2b+2};|\ve_3|_{k+2b+2})+ \mcL(|\ve_2|_{k+3b+2};|\ve_3|_{k+3b+2})\big)\cdot |\ve_3|_{a_0}. 
\end{align*}
This proves the lemma if we set $b_0:=3b+2$. 
\end{proof}
The following lemma summarizes all the estimates that we will need for the proof of Theorem \ref{14:11:52}. Together with Proposition \ref{10:57:29} below, this forms the input for the Nash-Moser algorithm. 

\begin{lem}\label{20:33:39} Suppose that $(M,I)$ is $2$- and $3$-convex and that $H^{0,2}(M,I)=0$. Then there exist $a_0,b_0\in \mathbb{Z}_{>0}$ and $\delta_0\in (0,1)$ with the following property.
For every $k\in\mathbb{Z}_{\geq0}$, every integrable $\ve\in\mcD$ satisfying $|\ve|_{a_0}<\delta_0$ and every $\xi\in \mcG$ satisfying $|\xi|_2<(1+|\sigma|_0)^{-1}$ 
we have
\begin{align} 
|(e^{d\xi}\cdot \ve)_i|_k & \leq |\ve_i|_k+C|\ve|_k|\xi|_{1}+C|\xi|_{k+1}, \hspace{15mm} \text{for} \ i=1,2,3,\label{18:14:22} \\
|\Phi(\ve)|_k  & \leq\mcL(|\ve_2|_{k+b_0};|\ve_3|_{k+b_0}), \label{20:48:36} \\
|(e^{d\Phi(\ve)}\cdot \ve)_3|_k &  \leq \big(\mcL(|\ve_1|_{k+b_0};|\ve_3|_{k+b_0})+\mcL(|\ve_2|_{k+b_0};|\ve_3|_{k+b_0})\big)\cdot |\ve_3|_{a_0}
\label{21:49:42}
\end{align}
\end{lem}
\begin{proof} 
We start with (\ref{18:14:22}): Since $|\xi|_1<(1+|\sigma|_0)^{-1}$ we know that $e^{d\xi}\cdot \ve$ is well-defined. The reason we require the bound on $|\xi|_2$ is to simplify the estimates a little in the following sense. Consider
\begin{align*}
\kappa:=1+d\xi^{2,0}(\sigma+\ve_1):T^{1,0}M\rightarrow T^{1,0}M,
\end{align*}
which is invertible by assumption. Then there exists a constant $C$ depending on $k$, $|\kappa^{-1}|_0$ and $|\kappa|_1$ (the latter depends on $|\xi|_2$) such that  
\begin{align}\label{17:24:09}
|\kappa^{-1}|_k\leq C(1+|\kappa|_k).
\end{align}
This can be proved by induction on $k$, using $\kappa\kappa^{-1}=1$ and (\ref{17:29:45}). Combining
$$\kappa^{-1}=1- d\xi^{2,0}(\sigma+\ve_1)\kappa^{-1},$$
(\ref{17:19:25}), (\ref{17:24:09}) and Lemma \ref{17:03:57}, we obtain Eq.(\ref{18:14:22}). 

The proof of equation (\ref{20:48:36}) follows from the definition of $\Phi$, Theorem \ref{15:35:57} 4) and Proposition \ref{13:43:23}.

Finally we turn to \eqref{21:49:42}. An explicit expression for $(e^{d\Phi(\ve)}\cdot \ve)_3$ is given in (\ref{14:48:01}). The first term in that expression was bounded in Lemma \ref{14:51:45}, while the remaining terms can be bounded more directly using again claim 4) from Theorem \ref{15:35:57} and Proposition \ref{13:43:23}.
\end{proof}
\begin{rem}
Whenever we use (\ref{18:14:22}) we will have a prescribed uniform bound on $|\xi|_2$, so that $C$ effectively only depends on $k$.  
\end{rem}
 The final tool needed for the algorithm is the concept of \textsl{smoothing operators}. 

\begin{prop}[\cite{MR656198}]\label{10:57:29}
There exists a family $\{S_t\}_{t>1}$ of endomorphisms on $\mcG=C^\infty(T^\ast M)$ satisfying  
\begin{align}
|S_t\xi|_p & \leq Ct^{p-q} |\xi|_q \label{18:14:42} \\
|(1-S_t)\xi|_q & \leq C t^{q-p} |\xi|_p  \label{11:03:30}
\end{align}
for every $p\geq q$, where $C$ is a constant depending only on $p$ and $q$. 
\end{prop}
\begin{rem}
Note that (\ref{18:14:42}) allows us to trade a high derivative norm $|\cdot |_p$ for a low derivative norm $| \cdot |_q$, while (\ref{11:03:30}) implies that $S_t\xi$ converges to $\xi$ as $t$ goes to infinity. As $t\rightarrow \infty$, the right hand side of (\ref{18:14:42}) blows up while the right hand side of (\ref{11:03:30}) goes to zero.
\end{rem}

\subsubsection*{The Nash-Moser algorithm}

We will now set up the algorithm that is required to prove Theorem \ref{14:11:52}. We will follow the strategy employed in \cite{MR2855089}, but since the setting is rather different we present things here independently. For a given constant $t_0>1$ we define the sequence $t_i$ via $t_{i+1}:=t_i^{3/2}$. Subsequently, for a small and integrable $\ve\in \mcD$ we define the sequence $(\ve^{(i)},\xi^{(i)})$ recursively via $\ve^{(0)}:=\ve$ and, for every $i\geq 0$, 
\begin{alignat}{3}
&\xi^{(i)}:=S_{t_{i}}\Phi(\ve^{(i)}),  \hspace{2cm} &&  \ve^{(i+1)}:=e^{d\xi^{(i)}}\cdot \ve^{(i)}. \label{11:56:23}
\end{alignat} 
Here $\Phi(\ve)$ is given by (\ref{14:14:54}) and $S_t$ denotes the smoothing operator of Proposition \ref{10:57:29}. The goal is to show that the series $\sum_{i} \xi^{(i)}$ converges in $\mcG$ to a smooth one-form $\xi$ with the property that $(e^{d\xi}\cdot \ve)_3=0$. We will establish this by means of two separate lemmas. The first will imply that our sequence is actually well-defined and that $\sum_{i} \xi^{(i)}$ converges with respect to a certain fixed $C^k$-norm. The second lemma will upgrade this convergence to all $C^k$-norms. 
\begin{lem}\label{22:15:42} With the set-up and notation of Lemma \ref{20:33:39}, define
\begin{align}\label{10:07:34}
B:=11+2b_0,	\quad \quad		l:=\max(a_0,(4b_0+2)(12+2b_0)+1).
\end{align} 
Then there exists $t_0>1$ with the following property: if $\ve^{(0)}\in \mcD$ is integrable and satisfies
\begin{align}\label{12:08:40}	
|\ve^{(0)}|_l < \frac{\delta_0}2 ,\hspace{15mm} |\ve^{(0)}|_{2l}<  t_0^B \hspace{7.5mm}\text{and} \hspace{7.5mm}  |\ve^{(0)}_3|_{l}< t_0^{-1},
\end{align}
then the sequence $(\ve^{(i)},\xi^{(i)})$ given by (\ref{11:56:23}) is well-defined and for all $i\in \mathbb{Z}_{\geq 0}$ we have
\begin{multicols}{2}
		\begin{enumerate}
			\item[(1)] $ |\ve^{(i)}|_l    < \delta_0\cdot \frac{i+1}{i+2} $
			\item[(2)] $|\ve^{(i)}|_{2l} < t_i^B$
			\item[(3)] $|\ve^{(i)}_3|_{l}  < t_i^{-1}$
			\item[(4)] $|\xi^{(i)}|_{l+b_0}  < t_i^{-1/2}$
		\end{enumerate}
\end{multicols}
\end{lem}
\begin{rem} Note that (1) and (4) imply that at each stage in the sequence we can apply the Neumann operator and hence the map $\Phi$, as well as Lemma \ref{20:33:39}. This means that our sequence is in fact well defined. Moreover, (4) implies that the sequence $\xi^{(i)}$ converges to zero fast enough so that the sum $\sum_i \xi^{(i)}$ converges in the $C^{l+b_0}$-topology. Property (3) then shows that the error terms $\ve^{(i)}_3$ converge to zero as desired. All of this comes at a price, for we have to allow $|\ve^{(i)}|_{2l}$ to grow at the rate indicated in (2).
\end{rem}
\begin{proof} We initially set (say) $t_0:=2$ and throughout the proof, which will proceed by induction on $i\geq 0$, we will enlarge $t_0$ a finite number of times. For $i=0$ only (4) needs proof, as the first three estimates are precisely the starting assumptions. We have
\begin{align*}
|\xi^{(0)}|_{l+b_0}=|S_{t_{0}}\Phi(\ve^{(0)})|_{l+b_0}
\leq C |\Phi(\ve^{(0)})|_{l+b_0}
\leq  \mcL(|\ve_2^{(0)}|_{l+2b_0};|\ve_3^{(0)}|_{l+2b_0}).
\end{align*}
Here we used (\ref{18:14:42}) and (\ref{20:48:36}). Since $l\geq 2b_0+2$ we can use (\ref{20:45:37}) to obtain
\begin{align}\label{21:56:26}
|\ve^{(0)}_2|_{l+2b_0}\leq C |\ve^{(0)}_2|_{l}^{(l-2b_0)/l}|\ve^{(0)}_2|^{2b_0/l}_{2l}\leq C(\tfrac{\delta_0}2)^{(l-2b_0)/l}t_0^{2b_0B/l}\leq C t_0^{2b_0B/l}
\end{align}
and, by a similar calculation, $|\ve^{(0)}_2|_{\lfloor (l+2b_0)/2\rfloor+1}\leq C$ and $|\ve^{(0)}_3|_{l+2b_0}\leq Ct_0^{(2b_0+2b_0B-l)/l}$. Note that we allow the constant $C$ (which depends on $l$) to change from one inequality to the next. By definition we have
\begin{align}\label{17:46:01}
\mcL(|\ve_2^{(0)}|_{l+2b_0};|\ve_3^{(0)}|_{l+2b_0})=&Poly(|\ve^{(0)}_2|_{\lfloor (l+2b_0)/2\rfloor+1})|\ve_3^{(0)}|_{l+2b_0}\\&+Poly(|\ve^{(0)}_2|_{\lfloor (l+2b_0)/2\rfloor+1})|\ve_2^{(0)}|_{l+2b_0}|\ve_3^{(0)}|_{\lfloor (l+2b_0)/2\rfloor+1},\nonumber
\end{align}
from which it follows that 
\begin{align}\label{15:35:34}
|\xi^{(0)}|_{l+b_0}\leq Ct_0^{(2b_0+2b_0B-l)/l}=Ct_0^{(4b_0+4b_0B-l)/2l}t_0^{-1/2}.
\end{align}
Since $l>4b_0(1+B)$, (4) holds if $t_0$ is chosen sufficiently large with respect to the constant $C$, which in turn only depends on $l$. This is the first update of $t_0$.

Now suppose that (1)-(4) hold for some $i\geq 0$, we will show that they hold for $i+1$ as well. For (1) we compute, using (\ref{18:14:22}),
\begin{align*}
|\ve^{(i+1)}|_l=|e^{d\xi^{(i)}}\cdot \ve^{(i)}|_l\leq &|\ve^{(i)}|_l+C(|\ve^{(i)}|_l|\xi^{(i)}|_{2}+|\xi^{(i)}|_{l+1})
 \leq  \delta_0 \cdot \frac{i+1}{i+2} +C t_{i}^{-1/2} ,
\end{align*}
which we would like to bound by $\delta_0\cdot \frac{i+2}{i+3}$. Slightly rewritten, this amounts to showing that
\begin{align*}
t_{i}^{1/2}\geq C(i+2)(i+3)/\delta_0.
\end{align*}
Since the right hand side is a polynomial in $i$ and $t_{i}^{\frac{1}{2}}=t_0^{\frac{1}{2}\cdot(\frac{3}{2})^{i}}$ grows exponentially in $i$, this estimate holds for all sufficiently large $i$ and therefore for all $i$ if $t_0$ is chosen sufficiently large with respect to $C$ (which depends only on $l$) and $\delta_0$. This is the second update of $t_0$.

For (2), we use (\ref{18:14:22}) to compute 
\begin{align}\label{21:42:22}
|\ve^{(i+1)}|_{2l}&\leq |\ve^{(i)}|_{2l}(1+C|\xi^{(i)}|_1)+C|\xi^{(i)}|_{2l+1}\leq C (t_{i}^B+|\xi^{(i)}|_{2l+1}).
\end{align}
Using (\ref{18:14:42}) and (\ref{20:48:36}), we obtain
\begin{align*}
|\xi^{(i)}|_{2l+1}
=|S_{t_{i}}\Phi(\ve^{(i)})|_{2l+1}
\leq & C t_{i}^{b_0+3}|\Phi(\ve^{(i)})|_{2l-b_0-2}
\leq t_{i}^{b_0+3}\cdot \mcL(|\ve^{(i)}_2|_{2l-2};|\ve^{(i)}_3|_{2l-2}) 
\end{align*}
Combining the induction hypothesis with (\ref{17:46:01}) we conclude that $|\xi^{(i)}|_{2l+1}\leq  Ct_{i}^{b_0+3+B}$, which we can substitute into (\ref{21:42:22}) to obtain
\begin{align*}
|\ve^{(i+1)}|_{2l} \leq C t_{i}^{b_0+3+B} = C t_{i+1}^{\frac{1}{3}(2b_0+6-B) } t_{i+1}^B.
\end{align*}
Since $B>2b_0+6$ we obtain (2) for $t_0$ sufficiently large with respect to $C$ (which depends only on $l$ and not on $i$). This is the third update of $t_0$. 

Next, to obain (3) we use (\ref{21:49:42}) to compute
\begin{align*}
|\ve_3^{(i+1)}|_l=|(e^{d\Phi(\ve^{(i)})}\cdot \ve^{(i)})_3|_l \leq& \big( \mcL(|\ve_1^{(i)}|_{l+b_0};|\ve_3^{(i)}|_{l+b_0})+\mcL(|\ve_2^{(i)}|_{l+b_0};|\ve_3^{(i)}|_{l+b_0})\big)\cdot |\ve_3^{(i)}|_{a_0}. 
\end{align*}
Using (\ref{20:45:37}) we can replace the norm $|\cdot |_{l+b_0}$ that appears in this expression by suitable powers of the norms $|\cdot |_{l}$ and $|\cdot |_{2l}$, just as we did in (\ref{21:56:26}). This gives
\begin{align*}
|\ve_3^{(i+1)}|_l\leq C t_{i}^{(b_0(B+1)-2l)/l}=Ct_{i+1}^{(2b_0(B+1)-l)/3l}t_{i+1}^{-1}
\end{align*}
which is bounded by $t_{i+1}^{-1}$ for $t_0$ sufficiently large, because $l>2b_0(B+1)$ (again, $C$ depends only on $l$). This is the fourth update on $t_0$. 

Finally, to derive (4) we use the by now proven inequalities (1), (2) and (3) at step $i+1$, and proceed as in the case of $i=0$ above to obtain 
\begin{align*}
|\xi^{(i+1)}|_{l+b}\leq C t_{i+1}^{(4b_0(1+B)-l)/2l}t_{i+1}^{-1/2}.
\end{align*}
Note that the constant $C$ here is the same as in (\ref{15:35:34}), for the bounds and the Leibniz polynomials that are involved only depend on $l$. In particular, because $t_{i+1}>t_0$ there is no need to further enlarge $t_0$ to obtain (4). 

In the end, to make the induction steps work we needed to enlarge $t_0$ four times with respect to data that depends only on the fixed integer $l$. This establishes the lemma. 
\end{proof}

At this point we have fixed $t_0$ and therefore the sequence $(\ve^{(i)},\xi^{(i)})$ (for given starting data $\ve^{(0)}$), we know that it converges in the $C^{l+b_0}$-topology with the desired limit and so we now have to upgrade this to convergence in the $C^\infty$-topology. We can not repeat the exact same strategy of the previous lemma because we lost the freedom of adjusting $t_0$, but fortunately it suffices at this point to obtain estimates that are weaker than those of Lemma \ref{22:15:42}.

\begin{lem}\label{11:50:30}
With the set-up and notation of Lemma \ref{22:15:42}, let $k\geq l$ be an integer and suppose that there exists a constant $D_k$ and an integer $d_k$ such that for all $i\geq d_k$ we have\footnote{Note that the case where $k=l$ and $d_l=0$ is precisely the conclusion of Lemma \ref{22:15:42}.}
\begin{align}\label{10:33:49}
|\ve^{(i)}|_{k}\leq D_k\cdot \frac{i+1}{i+2}, \hspace{10mm}  |\ve^{(i)}|_{2k}\leq D_kt_i^B \hspace{5mm} \text{and} \hspace{5mm} |\ve^{(i)}_3|_k\leq D_kt_i^{-1}.
\end{align}
Then there exist a constant $D_{k+1}$ and an integer $d_{k+1}> d_k$ such that for all $i\geq d_{k+1}$ we have
\begin{align*}
\text{(1)} & \hspace{3mm} |\xi^{(i)}|_{k+1+b_0}\leq t_i^{-1/2}  	\hspace{19mm}	\text{(3)}  \hspace{3mm}  |\ve^{(i)}|_{2(k+1)}\leq D_{k+1}t_i^B\\
\text{(2)} & \hspace{3mm} |\ve^{(i)}|_{k+1} \ \ \  \leq D_{k+1}\cdot \frac{i+1}{i+2} \hspace{1cm} \text{(4)}  \hspace{3mm}  |\ve^{(i)}_3|_{k+1} \ \ \ \leq D_{k+1} t_i^{-1}
\end{align*} 
\end{lem}
\begin{proof}
We start with (1). Using (\ref{18:14:42}) and (\ref{20:48:36}), we compute
\begin{align*}
|\xi^{(i)}|_{k+1+b_0}=|S_{t_{i}}\Phi(\ve^{(i)})|_{k+1+b_0}\leq &\mcL(|\ve^{(i)}_2|_{k+1+2b_0};|\ve^{(i)}_3|_{k+1+2b_0}).	
\end{align*}
As in the proof of Lemma \ref{22:15:42}, we can use (\ref{20:45:37}) to replace the norm $|\cdot |_{k+1+2b_0}$ by the norms $|\cdot |_{k}$ and $|\cdot |_{2k}$. If $i\geq d_k$, the hypothesis (\ref{10:33:49}) implies that  
\begin{align*}
|\xi^{(i)}|_{k+1+b_0}\leq &C t_{i}^{\frac{(2b_0+1)(1+B)}{k}-1}=C t_i^{\frac{(2b_0+1)(1+B)}{k}-\frac{1}{2}}t_i^{-1/2}.
\end{align*}
Since $k\geq l>(4b_0+2)(B+1)$ (c.f.\ (\ref{10:07:34})), the term that multiplies $t_i^{-1/2}$ above is a negative power of $t_i$, hence (1) will hold for $i\geq d_{k+1}$ when $d_{k+1}$ is sufficiently large with respect to the constant $C$ (which depends only on $k$). 

Next we consider (2). Using (\ref{18:14:22}) and the by now proven inequality (1), we obtain 
\begin{align}
|\ve^{(i)}|_{k+1}=|e^{d\xi^{(i-1)}}\cdot\ve^{(i-1)}|_{k+1}\leq &|\ve^{(i-1)}|_{k+1}(1+C|\xi^{(i-1)}|_1)+C |\xi^{(i-1)}|_{k+2} \nonumber
\\ \leq & |\ve^{(i-1)}|_{k+1}(1+Ct_{i-1}^{-1/2})+C t_{i-1}^{-1/2}. \label{10:21:31}
\end{align} 
If we would have a bound of the form $|\ve^{(i-1)}|_{k+1}\leq D_{k+1} \cdot \frac{i}{i+1}$, then the above inequality would imply $|\ve^{(i)}|_{k+1}\leq D_{k+1} \cdot \frac{i+1}{i+2}$, provided that
\begin{align}\label{10:20:58}
t_{i-1}^{1/2}\geq C(i+1)(i+2)\big(\frac{1}{D_{k+1}}+\frac{i}{i+1}\big). 
\end{align}
Therefore, we first enlarge $d_{k+1}$ sufficiently so that (\ref{10:20:58}), with $D_{k+1}$ temporarily set to $1$, holds for all $i\geq d_{k+1}$. Then, we enlarge $D_{k+1}$ sufficiently so that 
\begin{align}\label{10:22:13}
|\ve^{(d_{k+1})}|_{k+1}\leq D_{k+1}\cdot \frac{d_{k+1}+1}{d_{k+1}+2}. 
\end{align}
This procedure makes sense because enlarging $D_{k+1}$ does not spoil the bound (\ref{10:20:58}). Now (2) follows by induction on $i\geq d_{k+1}$, the base case being (\ref{10:22:13}) and the induction step using (\ref{10:21:31}) and (\ref{10:20:58}). Next, for (3) we compute
\begin{align}\label{10:53:42}
|\ve^{(i)}|_{2k+2}=|e^{d\xi^{(i-1)}}\cdot \ve^{(i-1)}|_{2k+2}\leq  |\ve^{(i-1)}|_{2k+2}(1+C|\xi^{(i-1)}|_2)+C|\xi^{(i-1)}|_{2k+3}.
\end{align}
Using (\ref{18:14:42}) and (\ref{20:48:36}), we deduce that
\begin{align*}
|\xi^{(i-1)}|_{2k+3}=|S_{t_{i-1}}\Phi(\ve^{(i-1)})|_{2k+3}\leq &C t_{i-1}^{5+b_0} |\Phi(\ve^{(i-1)})|_{2k-b_0-2}\\
\leq & t_{i-1}^{5+b_0} \mcL(|\ve_2^{(i-1)}|_{2k-2};|\ve_3^{(i-1)}|_{2k-2})\\
\leq & Ct_{i-1}^{5+b_0+B},
\end{align*}
where in the last step we used (\ref{10:33:49}). Substituting this into (\ref{10:53:42}) yields
\begin{align}\label{11:26:54}
|\ve^{(i)}|_{2k+2}\leq C( |\ve^{(i-1)}|_{2k+2} + t_i^{\frac{2}{3}(5+b_0+B)})=C( |\ve^{(i-1)}|_{2k+2} + t_i^{\frac{2}{3}(5+b_0-\frac{B}{2})} t_i^B).
\end{align}
Because $B>10+2b_0$ (c.f.\ (\ref{10:07:34})) we can ensure that $Ct_i^{\frac{2}{3}(5+b_0-\frac{B}{2})}\leq 1/2$ by enlarging $d_{k+1}$ sufficiently. Subsequently, we enlarge $D_{k+1}$ so that 
\begin{align}\label{11:10:41}
|\ve^{(d_{k+1})}|_{2k+2}\leq D_{k+1} t_{d_{k+1}}^B.
\end{align}
Then (3) follows by induction on $i\geq d_{k+1}$, the base case being (\ref{11:10:41}) and the inductive step being (\ref{11:26:54}) which becomes
\begin{align*}
|\ve^{(i)}|_{2k+2}\leq C( D_{k+1}t_i^{-B/3} + t_i^{\frac{2}{3}(5+b_0-\frac{B}{2})}) t_i^B\leq \tfrac{1}2(D_{k+1}+1)t_i^B\leq D_{k+1}t_i^B.
\end{align*}
Finally, for (4) we use (\ref{21:49:42}) to compute
\begin{align*}
|\ve_3^{(i)}|_{k+1}=|(e^{d\xi^{(i-1)}}\cdot\ve^{(i-1)})_3|_{k+1}&\leq \mcL(|\ve^{(i-1)}_1|_{k+b_0+1}+|\ve^{(i-1)}_2|_{k+b_0+1};|\ve^{(i-1)}_3|_{k+b_0+1})\cdot |\ve^{(i-1)}_3|_{a_0}.
\end{align*}
Using the by now familiar trick to replace $|\cdot |_{k+1+b_0}$ by powers of $|\cdot |_{k}$ and $|\cdot |_{2k}$, together with the hypothesis (\ref{10:33:49}), we obtain   
\begin{align*}
|\ve_3^{(i)}|_{k+1}\leq&\ Ct_i^{\frac{2}{3}(\frac{(1+b_0)(1+B)}{k} -2) }=C t_i^{\frac{1}{3}(\frac{2(1+b_0)(1+B)}{k} -1) }t_i^{-1}.
\end{align*}
As before, the term multiplying $t_i^{-1}$ is bounded with respect to $i$, and so (4) will hold if we enlarge $D_{k+1}$ sufficiently. In the end, we needed to enlarge $d_{k+1}$ and $D_{k+1}$ three times with respect to data that depends only on $k$. This establishes the lemma. 
\end{proof}

\subsubsection*{Proof of Theorem \ref{14:17:27}}\label{16:29:09}

\textsl{Proof.} Let $(M,I,\sigma)$ be a compact holomorphic Poisson manifold with boundary so that $(M,I)$ is $2$- and $3$-convex and for which $H^{0,2}(M,I)=0$. Let $a_0,b_0$ and $\delta_0$ be the associated constants of Lemma \ref{20:33:39}, define $l$ and $B$ by (\ref{10:07:34}) and let $t_0$ be the constant of Lemma \ref{22:15:42}. Define
\begin{align*}
a:=2l \hspace{15mm} \text{and}\hspace{15mm}  \delta:=\min(\tfrac{\delta_0}{2},t_0^{-1},t_0^B).
\end{align*}
Now let $\ve\in \mcD$ be integrable such that $|\ve|_a<\delta$. Then $\ve$ satisfies the hypotheses of Lemma \ref{22:15:42}, providing us with the corresponding sequence $(\ve^{(i)},\xi^{(i)})$ that satisfies the estimates of both Lemma \ref{22:15:42} and Lemma \ref{11:50:30}. In particular, the series
$$\xi:=\sum_{i=0}^\infty \xi^{(i)}$$
converges in $C^\infty(T^\ast M)$ and has the property that  
$$(e^{d\xi}\cdot \ve)_3=\lim_{n\rightarrow \infty} (e^{d\sum_{i=0}^n \xi^{(i)}}\cdot \ve^{(0)})_3=\lim_{n\rightarrow \infty} (\ve^{(n+1)})_3=0.	$$
This is what we set out to prove.  \qed


\subsection{Proof of the neighbourhood theorem}

\textsl{Proof of Theorem \ref{16:02:38}.}
Let $(Y,\tau)$ be an Abelian Poisson brane in $(M,\J)$. Recall that there is an induced complex structure on $Y$ and a holomorphic vector bundle structure on $NY$. In particular, the total space of $NY$ is a complex manifold whose complex structure we denote by $I$. By choosing a tubular embedding we may pull back $\J$ to a neighbourhood of $Y$ in $NY$, so we may assume without loss of generality that $M=NY$ and that $Y$ is the zero section. For $t\in [0,1]$ we denote by $m_t:M\rightarrow M$ the fiber-wise rescaling, which for $t=0$ equals the projection to $Y$. Our first step is to show that $\J_t:=m_t^\ast(\J)$ converges, as $t\rightarrow 0$, to a holomorphic Poisson structure on $M$ whose complex structure agrees with $I$.  

Let us choose a convenient splitting of the exact Courant algebroid $E$ as follows. Consider a splitting $E=TM\oplus T^\ast M$ so that the induced splitting $E_Y=TY\oplus T^\ast Y$ coincides with the splitting induced by the brane structure $\tau$. In particular, the three-form $H$ on $M$ satisfies $i^\ast H=0$, where $i:Y\hookrightarrow M$ denotes the inclusion. Since $M$ is a vector bundle over $Y$, this implies that $H=dB$ for some two-form $B$ that vanishes on $Y$, so we may change the original splitting by $B$ to arrange for $H=0$. Next, we write $\J$ in this splitting as
\begin{align}\label{18:29:40}
\J=\begin{pmatrix} -A & \pi \\ \beta & A^\ast	\end{pmatrix}
\end{align}
for an endomorphism $A:TM\rightarrow TM$, bi-vector $\pi:T^\ast M\rightarrow TM$ and two-form $\beta:TM\rightarrow T^\ast M$. Because $Y$ is a Poisson brane we have:
\begin{align}\label{19:44:13}
A^\ast(N^\ast Y)=N^\ast Y, \quad\quad \pi(N^\ast Y)=0, \quad\quad i^\ast \beta=0.
\end{align}
Here the first two equations are a consequence of $\J( N^\ast Y)=N^\ast Y$. The third equation follows from the fact that the induced generalized complex structure on $TY\oplus T^\ast Y=(TY\oplus T^\ast M)/N^\ast Y$ is holomorphic Poisson. 
\begin{lem}
The generalized complex structure $\J_t$ converges, as $t\rightarrow 0$, to a generalized complex structure $\J_0$ which is given by a holomorphic Poisson structure.
\end{lem}
\begin{proof} It suffices to look at a single coordinate chart. Consider a chart in $Y$ with coordinates $(y^1,\ldots,y^k)$ over which $M=NY$ is trivialized as a vector bundle, with fiber coordinates $(x^1,\ldots,x^l)$. We will abbreviate the resulting coordinates on $M$ by $(y,x):=(y^1,\ldots,y^k,x^1,\ldots,x^l)$, dropping the indices to enhance readability. Then $m_t$ is given by $m_t(y,x)=(y,tx)$ while the tensors $A,\pi,\beta$ can be decomposed as 
\begin{align*}
A_{(y,x)}=&A^{xx}_{(y,x)}dx\otimes \del_x+A^{xy}_{(y,x)}dx\otimes \del_y+A^{yx}_{(y,x)}dy\otimes \del_x +A^{yy}_{(y,x)}dy\otimes \del_y,\\
\pi_{(y,x)}=&\pi^{xx}_{(y,x)} \del_x\wedge \del_x+\pi^{xy}_{(y,x)}\del_x\wedge \del_y+\pi^{yy}_{(y,x)}\del_y\wedge \del_y,\\
\beta_{(y,x)}=&\beta^{xx}_{(y,x)}dx\wedge dx+\beta^{xy}_{(y,x)}dx\wedge dy+\beta^{yy}_{(y,x)}dy\wedge dy.
\end{align*}  
The indices have been dropped in these expressions, and the subscript $(y,x)$ denotes a point in $M$. With respect to these decompositions, (\ref{19:44:13}) reduces to
$$A^{yx}_{(y,0)}=0, \quad\quad \pi^{xx}_{(y,0)}=\pi^{xy}_{(y,0)}=0, \quad\quad \beta^{yy}_{(y,0)}=0.		$$ 
Furthermore, the Lie algebra structure on $N^\ast Y$ is given by 
\begin{align*}
[dx^i,dx^j]=\sum_{r=1}^l(\del_{x^r}\pi^{xx,ij})_{(y,0)}dx^r,
\end{align*}
so the fact that $N^\ast Y$ is Abelian is equivalent to $(\del_{x}\pi^{xx})_{(y,0)}=0$ (again dropping indices). All together, these five conditions imply that the rescalings of $A$, $\pi$ and $\beta$, which are given by 
\begin{align}
(m_t^\ast A)_{(y,x)}=&A^{xx}_{(y,tx)}dx\otimes \del_x+tA^{xy}_{(y,tx)}dx\otimes \del_y+\frac{1}{t}A^{yx}_{(y,tx)}dy\otimes \del_x +A^{yy}_{(y,tx)}dy\otimes \del_y, \nonumber\\
(m_t^\ast \pi)_{(y,x)}=&\frac{1}{t^2}\pi^{xx}_{(y,tx)} \del_x\wedge \del_x+\frac{1}{t}\pi^{xy}_{(y,tx)}\del_x\wedge \del_y+\pi^{yy}_{(y,tx)}\del_y\wedge \del_y,\nonumber \\
(m_t^\ast \beta)_{(y,x)}=&t^2\beta^{xx}_{(y,tx)}dx\wedge dx+t\beta^{xy}_{(y,tx)}dx\wedge dy+\beta^{yy}_{(y,tx)}dy\wedge dy, \label{11:23:45}
\end{align}  
converge as $t\rightarrow 0$, with $m_t^\ast(\beta)$ converging to zero. That means that $\J_t$ converges, as a tensor, to a holomorphic Poisson structure $\J_0$ on $M$. 
\end{proof}
It now remains to show that the complex structure underlying $\J_0$, which we denote by $I_0$, agrees with $I$. Recall that the latter is obtained from the holomorphic structure on $NY$, which is induced by the generalized complex structure $\J$. For $t>0$ we have $\J_t=m_t^\ast(\J)$, and we know that $m_t$ is a diffeomorphism which is the identity on $Y$ and whose derivative is the identity on $NY$. As such, every $\J_t$ for $t>0$ induces the same holomorphic structure on $NY$ and therefore so does the limiting structure $\J_0$. Consequently, the complex structure $I_0$ on $M$ induces the holomorphic structure on $NY$ from which $I$ is derived. In addition, we have $I_0=m_t^\ast(I_0)$ for all $t>0$ because $\J_0$ is per construction scale-invariant. The fact that $I=I_0$ then follows from the following general lemma.
\begin{lem}
Let $Y$ be a real manifold and $p:M\rightarrow Y$ a real vector bundle over $Y$, with $m_t:M\rightarrow M$ denoting the fiberwise rescaling map. Let $I_0$ be a scale-invariant complex structure on $M$ for which $Y$ is a complex submanifold. Then $I_0$ coincides with the holomorphic vector bundle structure that it induces on $M$.   
\end{lem}
\begin{proof}
It suffices to show that $(M,I_0)$ is a holomorphic vector bundle, which we can prove by constructing local holomorphic trivializations over $Y$. For any point in $Y$ we can pick an open neighbourhood $U\subset M$ together with holomorphic functions $f_1,\ldots,f_k\in \mathcal{O}(U)$ that cut out $Y$ as a complex submanifold. Since $I_0$ is scale-invariant, we know that $m_t^\ast(f_1),\ldots,m_t^\ast(f_k)$ are also holomorphic on $m_t^{-1}(U)$ and also cut out $Y$ as a complex submanifold. The limits of these functions as $t\rightarrow 0$ exist, and are linear holomorphic functions on $p^{-1}U\subset M$. This constitutes the desired holomorphic trivialization $(f_1,\ldots,f_k):p^{-1}(U)\rightarrow U\times \C^k$.      
\end{proof}

Right now we know that the family of rescalings $\J_t=m_t^\ast(\J)$ converges to a holomorphic Poisson structure $\J_0$ whose underlying complex structure equals $I$, the complex structure associated canonically with the brane $Y$. One of the hypotheses on $Y$ in Theorem \ref{16:02:38} is that there exists a neighbourhood $U$ in $M$ (recall that $M=NY$ here) such that $(U,I)$ is 2- and 3- convex and for which $H^{0,2}(U,I)=0$. 
These are exactly the hypotheses of Theorem \ref{14:11:52}, hence for $t_0>0$ sufficiently small we obtain an exact two-form $B=d\xi$ that transforms $\J_{t_0}$ into a holomorphic Poisson structure on $U$. Since $\J_{t_0}=m_{t_0}^\ast(\J)$, this implies that 
$$(m_{t_0})_\ast B=d(m_{t_0})_\ast \xi$$ 
transforms $\J$ into a holomorphic Poisson structure on the neighbourhood $V:=m_{t_0}(U)$. This completes the proof of Theorem \ref{16:02:38}.
\qed

\begin{rem}
In the proof we showed that the limit $\J_0$ is given by a holomorphic Poisson structure whose complex structure agrees with the canonical holomorphic structure on $NY$. It is then natural to wonder whether a similar statement holds for the Poisson structure underlying $\J_0$. It turns out that an Abelian Poisson brane $Y$ carries a canonical holomorphic Poisson structure on $NY$, which then can be shown to agree with the limit $\J_0$. In other words, the limit $\J_0$ that we obtain by a scaling argument using a choice of tubular neighbourhood is actually canonically associated to $Y$ from the fact that it is an Abelian Poisson brane. We will not give any details here because they are not relevant for the proof. This is because the Poisson bivector plays no role in the rigidity statement (Theorem \ref{14:11:52}) and the final claim is only that $\J$ is equivalent to some holomorphic Poisson structure, which may very well be different from $\J_0$. If we wanted to obtain a normal form, i.e.\ show that $\J$ is actually equivalent to $\J_0$ itself, we would have to study the bivector as well (in particular we would need to impose restrictions on cohomology groups that involve the bivector).
\end{rem}

\subsubsection*{The case of a point}

Here we will prove Theorem \ref{16:05:53}, which deals with the special case when $Y$ is a point in the complex locus of $\J$. The proof of Theorem \ref{16:02:38} above only works when the normal bundle is Abelian, which in turn was needed to show that the sequence $m_t^\ast(\J)$ converges as $t$ goes to zero. When $Y$ is a point we can get around this as follows. For $t>0$ we define a map $\lambda_t$ on $\T M$ by
\begin{align*}
\lambda_t(X+\xi):= tX+\tfrac{1}{t}\xi.
\end{align*}
For $u,v\in \T M$ we have
$$\langle \lambda_t(u),\lambda_t(v)\rangle = \langle u,v \rangle, \quad \quad \lb \lambda_t(u),\lambda_t(v)\rb_H = t\lambda_t(\lb u,v \rb_{t^2H}),$$
which shows that $\lambda_t$ is never a symmetry of $\T M$. However, in the special case that $H=0$ we do see that $\lambda_t(L)$ is Dirac for any Dirac structure $L$. So when $H=0$ we can apply $\lambda_t$ to a generalized complex structure to obtain another one. The last main feature of $\lambda_t$ is that it preserves holomorphic Poisson structures, specifically we have
\begin{align*}
\lambda_t(\J_{(I,\sigma)})= \J_{(I,t^2\sigma)}.
\end{align*}
\textsl{Proof of Theorem \ref{16:05:53}.}
Similarly to the proof of Theorem \ref{16:02:38}, we first use a chart to reduce to the case where $M=\C^n$ and $Y$ is the origin, with a splitting in which $H=0$, and write $\J$ as in (\ref{18:29:40}). We know that $m_t^\ast \pi$ does not converge in general, so we consider the following family of generalized complex structures instead:
\begin{align*}
\widetilde{J}_t:= \lambda_{{\sqrt{t}}}(m_t^\ast\J) = \begin{pmatrix} -m_t^\ast A & t\cdot m_t^\ast \pi  \\ \tfrac{1}{t}\cdot m_t^\ast \beta & (m_t^\ast A)^\ast	\end{pmatrix}.
\end{align*}
If $(x)$ denotes a set of coordinates around $Y$ (dropping indices again) we have
\begin{align*}
(t\cdot m_t^\ast \pi)_{(x)}=&\frac{1}{t}\pi^{xx}_{(tx)} \del_x\wedge \del_x,		\quad \quad\quad  		(\tfrac{1}{t}\cdot m_t^\ast \beta)_{(x)}=t\beta^{xx}_{(tx)}dx\wedge dx.
\end{align*}
Since $Y$ lies in the complex locus we have $\pi^{xx}_{(0)}=0$, which means that $\widetilde{J}_t$ converges (as $t$ goes to zero) to a holomorphic Poisson structure $\widetilde{\J}_0$. The underlying complex structure of $\widetilde{\J}_0$ is the standard one on $\C^n$, so the unit ball $\mathbb{D}\subset \C^n$ is 2- and 3- convex and $H^{0,2}(\mathbb{D})=0$. Once again we are in the position to apply our rigidity result (Theorem \ref{14:11:52}), implying that $\widetilde{\J}_{t_0}$ is equivalent to a holomorphic Poisson structure via some exact two-form $B\in \Omega^2(\mathbb{D})$. Since $\lambda_t$ preserves holomorphic Poisson structures, this implies that $\J$ itself is equivalent to a holomorphic Poisson structure on $m_t(\mathbb{D})$ via the two-form $(m_t)_\ast(tB)$. \qed

\begin{rem}
Note that this strategy does not work when $Y$ is not a point. One can still consider $\widetilde{J}_t$ which does converge as $t$ goes to zero, but the limit is not guaranteed to be holomorphic Poisson. The reason is that $\tfrac{1}{t}\cdot m_t^\ast \beta$ does not necessarily converge to zero when $Y$ is not a point, as can be seen from (\ref{11:23:45}).
\end{rem}

\begin{rem}
Our arguments for the noncompact case are very reminiscent of a number of other results where Nash--Moser techniques have been used. Particularly close are Hamilton's papers on deformations of complex structures \cite{MR477158,MR594711} which not only use the same techniques but also deal with the same geometric structure and operators. One might wonder if we could not use one of the general results presented there to obtain a proof of our main theorem  without re-running the analytical argument. Yet, as we  remarked in \ref{rem:24-08-22},  our framework is distinct from a standard deformation theory problem and it does not fit Hamilton's setup.
\end{rem}

\bibliographystyle{hyperamsplainnodash}
\addcontentsline{toc}{section}{References}
\bibliography{references}

\end{document}